\documentclass[12pt,twoside]{amsart}
\usepackage{amssymb}
\usepackage{amscd}
\usepackage{amsmath}
\usepackage{xypic}

\title[Abundance Conjecture]{Log pluricanonical 
representations and abundance conjecture}
\author{Osamu Fujino} 
\author{Yoshinori Gongyo}

\date{2012/4/5, version 1.78}
\address{Department of Mathematics, Faculty of Science, 
Kyoto University, Kyoto 606-8502, Japan}
\email{fujino@math.kyoto-u.ac.jp}

\address{Graduate School of Mathematical Sciences, 
The University of Tokyo, 3-8-1 Komaba, 
Meguro, Tokyo, 153-8914 Japan.}
\email{gongyo@ms.u-tokyo.ac.jp}


\newcommand{\Exc}[0]{\operatorname{Exc}}
\newcommand{\Supp}[0]{\operatorname{Supp}}
\newcommand{\Bir}[0]{\operatorname{Bir}}
\newcommand{\Bim}[0]{\operatorname{Bim}}
\newcommand{\Nklt}[0]{\operatorname{Nklt}}
\newcommand{\id}[0]{\operatorname{id}}
\newcommand{\Aut}[0]{\operatorname{Aut}}

\newtheorem{thm}{Theorem}[section]
\newtheorem{lem}[thm]{Lemma}
\newtheorem{cor}[thm]{Corollary}
\newtheorem{prop}[thm]{Proposition}
\newtheorem{conj}[thm]{Conjecture}

\theoremstyle{definition}
\newtheorem{ex}[thm]{Example}
\newtheorem{defn}[thm]{Definition}
\newtheorem{rem}[thm]{Remark}
\newtheorem*{ack}{Acknowledgments}       
\newtheorem{say}[thm]{}
\newtheorem{case}{Case}
\begin{document}

\maketitle 

\begin{abstract}
We prove the finiteness of 
log pluricanonical representations 
for projective log canonical pairs 
with semi-ample log canonical divisor. 
As a corollary, we obtain that 
the log canonical divisor of a projective semi log canonical 
pair is semi-ample if and only if so is the log canonical divisor 
of its normalization. 
We also treat many other applications.  

\end{abstract}

\tableofcontents

\section{Introduction}

The following theorem is one of the main results of this paper (cf.~Theorem 
\ref{semi-ample1}). 
It is a solution of the conjecture raised in \cite{fujino-abundance} 
(see \cite[Conjecture 3.2]{fujino-abundance}). 
For the definition of the {\em{log pluricanonical 
representation}} $\rho_m$, see 
Definitions \ref{B-bir} and \ref{B-repre} below. 

\begin{thm}[{cf.~\cite[Section 3]{fujino-abundance}, 
\cite[Theorem B]{gongyo-aban}}]\label{main-thm} 
Let $(X,\Delta)$ be a projective log canonical pair. 
Suppose 
that $m(K_X+\Delta)$ is 
Cartier and that $K_X+\Delta$ is semi-ample.  
Then $\rho_{m}(\mathrm{Bir}(X,\Delta))$ is a finite group. 
\end{thm}

In Theorem \ref{main-thm}, 
we do not have to assume that $K_X+\Delta$ is semi-ample when 
$K_X+\Delta$ is big (cf.~Theorem \ref{finiteness1}). 
As a corollary of this fact, 
we obtain the finiteness of $\Bir (X, \Delta)$ when 
$K_X+\Delta$ is big. 
It is an answer to the question raised by Cacciola and 
Tasin. 

\begin{thm}[cf.~Corollary \ref{finiteness birational auto}]
Let $(X, \Delta)$ be a projective log canonical pair such that 
$K_X+\Delta$ is big. 
Then 
$\Bir (X, \Delta)$ is a finite group. 
\end{thm}

In the framework of \cite{fujino-abundance}, 
Theorem \ref{main-thm} will play important roles in the 
study of Conjecture \ref{abun} (see \cite{fujita}, \cite{AFKM}, \cite{k1}, 
\cite{kemamc}, \cite{fujino-abundance}, \cite{fujino-surface}, 
\cite{gongyo-aban}, and so on). 

\begin{conj}[(Log) abundance conjecture]\label{abun} 
Let $(X,\Delta)$ be a 
projective semi log canonical pair such that $\Delta$ is a 
$\mathbb{Q}$-divisor. 
Suppose that $K_X+\Delta$ is nef. Then $K_X+\Delta$ is semi-ample.
\end{conj}

Theorem \ref{main-thm} was settled for surfaces in \cite[Section 
3]{fujino-abundance} and for the case where $K_X+\Delta\sim _{\mathbb Q}0$ 
by \cite[Theorem B]{gongyo-aban}. 
In this paper, to carry out the proof of Theorem \ref{main-thm}, 
we introduce the notion of {\em{$\widetilde B$-birational 
maps}} and {\em{$\widetilde B$-birational representations}} 
for sub kawamata log terminal pairs, which is new and is indispensable 
for generalizing the arguments in \cite[Section 3]{fujino-abundance} 
for higher dimensional log canonical pairs. 
For the details, see Section \ref{sec3}. 

By Theorem \ref{main-thm}, 
we obtain a key result. 

\begin{thm}[{cf.~Proposition \ref{prop3}}]\label{main-thm2} 
Let $(X, \Delta)$ be a projective semi log canonical pair.  
Let $\nu:X^{\nu}\to X$ be the normalization. 
Assume that $K_{X^\nu}+\Theta=\nu^*(K_X+\Delta)$ is semi-ample. 
Then $K_X+\Delta$ is semi-ample. 
\end{thm}

By Theorem \ref{main-thm2}, Conjecture \ref{abun} is reduced to 
the problem for log canonical pairs. 
After we circulated this paper, Hacon and Xu proved a relative version of Theorem \ref{main-thm2} 
by using Koll\'ar's gluing theory 
(cf.~\cite{hacon-xu-slc}). For the details, see Subsection \ref{subsec41} below. 

Let $X$ be a smooth projective $n$-fold. By our 
experience on the low-dimensional 
abundance conjecture, we think that we need the abundance 
theorem for projective semi log canonical pairs in dimension $\leq n-1$ in order 
to prove the abundance conjecture for $X$. 
Therefore, Theorem \ref{main-thm2} seems to be an important step 
for the inductive approach to the abundance conjecture. 
The general strategy for proving the abundance conjecture 
is explained in the introduction of \cite{fujino-abundance}. 
Theorem \ref{main-thm2} is a complete solution of Step (v) in \cite[0.~Introduction]
{fujino-abundance}. 

As applications of Theorem \ref{main-thm2} and \cite[Theorem 1.1]{fujino-bpf}, 
we have the following useful theorems. 

\begin{thm}[{cf.~Theorem \ref{thm1}}]\label{thm-a}
Let $(X, \Delta)$ be a projective log canonical pair such that 
$\Delta$ is a $\mathbb Q$-divisor. 
Assume that $K_X+\Delta$ is nef and 
log abundant. 
Then $K_X+\Delta$ is semi-ample. 
\end{thm}

It is a generalization of the well-known theorem 
for kawamata log terminal pairs (see, 
for example, \cite[Corollary 2.5]{fujino-kawamata}). 
Theorem \ref{thm-b} may be easier to understand than Theorem 
\ref{thm-a}. 

\begin{thm}[{cf.~Theorem \ref{thm2}}]\label{thm-b}
Let $(X, \Delta)$ be an $n$-dimensional projective log canonical pair 
such that $\Delta$ is a $\mathbb Q$-divisor. 
Assume that 
the abundance conjecture holds for 
projective divisorial log terminal pairs in dimension $\leq n-1$. 
Then $K_X+\Delta$ is semi-ample if and only if 
$K_X+\Delta$ is nef and abundant. 
\end{thm}

We have many other applications. 
In this introduction, we explain only one of 
them. It is a generalization of \cite[Theorem 0.1]{fukuda2} 
and \cite[Corollary 3]{ckp}. It also contains 
Theorem \ref{thm-a}. For a further generalization, see Remark \ref{rem416}. 
\begin{thm}[{cf.~Theorem \ref{thm415}}]
Let $(X, \Delta)$ be a projective log canonical pair 
and let $D$ be a $\mathbb Q$-Cartier $\mathbb Q$-divisor 
on $X$ such that 
$D$ is nef and log abundant with respect to 
$(X, \Delta)$. 
Assume that $K_X+\Delta\equiv D$. 
Then $K_X+\Delta$ is semi-ample. 
\end{thm}

The reader can find many applications and 
generalizations in Section \ref{sec1}. In Section \ref{sec5}, we will discuss the relationship 
among the various conjectures in the minimal model program. 
Let us recall the following two important conjectures. 

\begin{conj}[Non-vanishing conjecture]\label{intro-nonvani}
Let $(X, \Delta)$ be a projective 
log canonical pair such that $\Delta$ is an $\mathbb R$-divisor. 
Assume that $K_X+\Delta$ is pseudo-effective. 
Then there exists an effective $\mathbb R$-divisor $D$ on 
$X$ such that $K_X+\Delta\sim_{\mathbb{R}} D$. 
\end{conj}

By \cite[Section 8]{dhp} and \cite{g4}, 
Conjecture \ref{intro-nonvani} can be reduced to the case when 
$X$ is a smooth projective variety and 
$\Delta=0$ by using the global ACC conjecture and 
the ACC for log canonical thresholds 
(see \cite[Conjecture 8.2 and Conjecture 8.4]{dhp}). 

\begin{conj}[{Extension conjecture for divisorial log terminal pairs 
(cf.~\cite[Conjecture 1.3]{dhp})}]\label{intro-ext}
Let $(X, \Delta)$ be an $n$-dimensional projective 
divisorial log terminal pair such that $\Delta$ is a $\mathbb Q$-divisor, 
$\llcorner  \Delta \lrcorner =S$, $K_X+\Delta$ is nef, 
and $K_X+\Delta\sim _{\mathbb Q}D\geq 0$ where $S\subset\Supp D$.
Then $$H^0(X,\mathcal O _X(m(K_X+\Delta)))\to H^0(S,  \mathcal O _S(m(K_X+\Delta)))  $$
is surjective for all sufficiently divisible integers $m\geq 2$. 
\end{conj}

Note that Conjecture \ref{intro-ext} holds true when 
$K_X+\Delta$ is semi-ample (cf.~Proposition \ref{inj imply ext}). 
It is an easy consequence of a cohomology injectivity theorem. 
We also note that Conjecture \ref{intro-ext} 
is true if $(X, \Delta)$ is purely log terminal (cf.~\cite[Corollary 1.8]{dhp}). 
The following theorem is one of the main results of 
Section \ref{sec5}. It is a generalization of \cite[Theorem 1.4]{dhp}. 

\begin{thm}[{cf.~\cite[Theorem 1.4]{dhp}}]\label{thm110} 
Assume that {\em{Conjecture \ref{intro-nonvani}}} and {\em{Conjecture \ref{intro-ext}}} 
hold true in dimension $\leq n$. 
Let $(X, \Delta)$ be an $n$-dimensional 
projective divisorial log terminal pair such that 
$K_X+\Delta$ is pseudo-effective. 
Then 
$(X, \Delta)$ has a good minimal model. 
In particular, if $K_X+\Delta$ is nef, then $K_X+\Delta$ is 
semi-ample. 
\end{thm}

By our inductive treatment of Theorem \ref{thm110}, 
Theorem \ref{main-thm2} plays a crucial role. 
Therefore, Theorem \ref{main-thm} is indispensable for Theorem \ref{thm110}. 

We summarize the contents of this paper. 
In Section \ref{prel}, we collects some basic notations and results. 
Section \ref{sec3} is the main part of this paper. 
In this section, we prove Theorem \ref{main-thm}. 
We divide the proof into the three steps:~sub kawamata 
log terminal pairs in 
\ref{kltcase}, log canonical 
pairs with big log canonical divisor in \ref{general 
type case}, and log canonical 
pairs with semi-ample log canonical divisor in \ref{semi-ample 
case}. 
Section \ref{sec1} contains various applications 
of Theorem \ref{main-thm}. 
They are related to the abundance conjecture:~Conjecture \ref{abun}. 
In Subsection \ref{mis}, we generalize the main theorem 
in \cite{fukuda2} (cf.~\cite[Corollary 3]{ckp}), 
the second author's result in \cite{gongyo-weak}, and so on. 
In Section \ref{sec5}, 
we discuss the relationship among the various conjectures 
in the minimal model program.  

\begin{ack}
The first author was partially supported by The Inamori
Foundation and by the Grant-in-Aid for Young Scientists (A) 
$\sharp$20684001 from JSPS. 
He is grateful to Professors Yoichi Miyaoka, 
Shigefumi Mori, Noboru Nakayama for giving him many useful comments 
during the preparation of \cite{fujino-abundance}, which was 
written as his master's thesis under the supervision of Professor Shigefumi 
Mori. This paper is a sequel of \cite{fujino-abundance}. 
The second author was 
partially supported by the Research 
Fellowships of the Japan Society for the Promotion of Science for Young Scientists.
He thanks Professors Caucher Birkar, Paolo Cascini, 
Mihai P\u{a}un, and Hajime Tsuji for comments and discussions. He wishes to thank 
Salvatore Cacciola and Luca Tasin for their question. 
He also thanks Professor Claire Voisin and Institut de Math\'ematiques de Jussieu for 
their hospitality. 
The main idea of this paper was obtained when the both authors stayed at CIRM. 
They are grateful to it for its hospitality. 
\end{ack}

We will work over $\mathbb C$, the complex number field, throughout this paper. 
We will freely use the standard notations in \cite{km}. 

\section {Preliminaries}\label{prel}

In this section, we collects some basic notations and 
results. 

\begin{say}[Convention]Let $D$ be 
a Weil divisor on a normal variety $X$. 
We sometimes 
simply write 
$H^0(X, D)$ to denote $H^0(X, \mathcal O_X(D))$. 
\end{say}

\begin{say}[{$\mathbb Q$-divisors}]
For a $\mathbb Q$-divisor $D=\sum _{j=1}^r d_j D_j$ on a normal 
variety $X$ such that 
$D_j$ is a prime divisor for every $j$ and 
$D_i\ne D_j$ for $i\ne j$, we define the {\em{round-down}} 
$\llcorner D\lrcorner =\sum _{j=1}^r\llcorner d_j \lrcorner D_j$, where 
for every rational number $x$, 
$\llcorner x\lrcorner$ is the integer defined by $x-1<\llcorner x\lrcorner \leq x$. 
We put 
$$
D^{=1}=\sum _{d_j=1}D_j. 
$$
We note that $\sim _{\mathbb Z}$ ($\sim$, for short) 
denotes the {\em{linear equivalence}} 
of divisors. 
We also note 
that $\sim _{\mathbb Q}$ (resp.~$\equiv$) 
denotes the {\em{$\mathbb Q$-linear equivalence}} 
(resp.~{\em{numerical equivalence}}) of $\mathbb Q$-divisors. 
Let $f:X\to Y$ be a morphism and let $D_1$ and $D_2$ be 
$\mathbb Q$-Cartier $\mathbb Q$-divisors on $X$. 
Then $D_1\sim _{\mathbb Q, Y}D_2$ means that 
there is a $\mathbb Q$-Cartier $\mathbb Q$-divisor $B$ on $Y$ such that 
$D_1\sim _{\mathbb Q}D_2+f^*B$. We can also treat 
$\mathbb R$-divisors similarly. 
\end{say}

\begin{say}[Log resolution]
Let $X$ be a normal variety and let $D$ be an 
$\mathbb R$-divisor on $X$. 
A {\em{log resolution}} 
$f:Y\to X$ means that 
\begin{itemize}
\item[(i)] $f $ is a proper birational morphism, 
\item[(ii)] $Y$ is smooth, and 
\item[(iii)] $\Exc (f) \cup \Supp f_*^{-1}D$ is a 
simple normal crossing divisor on $Y$, where 
$\Exc (f)$ is the {\em{exceptional locus}} of $f$. 
\end{itemize}
\end{say}

We recall the notion of singularities of pairs. 

\begin{defn}[Singularities of pairs] 
Let $X$ be a normal variety and let $\Delta$ be an $\mathbb{R}$-divisor 
on $X$ such that $K_X+\Delta$ is $\mathbb{R}$-Cartier. 
Let $\varphi:Y\rightarrow X$ be a log resolution of $(X,\Delta)$. 
We set $$K_Y=\varphi^*(K_X+\Delta)+\sum a_iE_i,$$ where $E_i$ is a 
prime divisor on $Y$ for every $i$.
The pair $(X,\Delta)$ is called 
\begin{itemize}
\item[(a)] \emph{sub kawamata log terminal} $($\emph{subklt}, 
for short$)$ if $a_i > -1$ for all $i$, or
\item[(b)]\emph{sub log canonical} $($\emph{sublc}, for short$)$ if $a_i \geq -1$ for all $i$.
\end{itemize}
If $\Delta$ is effective and $(X, \Delta)$ is subklt (resp.~sublc), 
then we  simply call it \emph{klt} (resp.~\emph{lc}). 

Let $(X, \Delta)$ be an lc pair. 
If there is a log resolution $\varphi:Y\to X$ of $(X, \Delta)$ such that 
$\Exc (\varphi)$ is a divisor and that 
$a_i>-1$ for every $\varphi$-exceptional 
divisor $E_i$, then 
the pair $(X, \Delta)$ is called 
{\em{divisorial log terminal}} ({\em{dlt}}, for short). 
\end{defn}

Let us recall {\em{semi log canonical pairs}} and 
{\em{semi divisorial log terminal pairs}} 
(cf.~\cite[Definition 1.1]{fujino-abundance}). 
For the details of these pairs, see \cite[Section 1]{fujino-abundance}. 
Note that the notion of semi divisorial log terminal pairs 
in \cite[Definition 5.17]{kollar-book} is different from ours. 
 
\begin{defn}[Slc and sdlt] 
Let $X$ be a reduced $S_2$ scheme. 
We assume that it is pure $n$-dimensional 
and normal crossing in codimension one. 
Let $\Delta$ be an effective $\mathbb Q$-divisor 
on $X$ such that $K_X+\Delta$ is $\mathbb Q$-Cartier. 
We assume that $\Delta=\sum _i a_i \Delta_i$ where 
$a_i \in \mathbb Q$ and $\Delta_i$ 
is an irreducible codimension one 
closed subvariety of $X$ such that 
$\mathcal O_{X, \Delta_i}$ is a DVR for every $i$. 
Let $X=\cup _i X_i$ be the irreducible decomposition and 
let $\nu:X^\nu:=\amalg _i X_i^\nu\to X=\cup _i X_i$ be the normalization. 
A $\mathbb Q$-divisor $\Theta$ on $X^\nu$ 
is defined by $K_{X^\nu}+\Theta=\nu^*(K_X+\Delta)$ and 
a $\mathbb Q$-divisor $\Theta_i$ on $X_i^\nu$ by $\Theta_i:=\Theta|_{X_i^\nu}$. 
We say that $(X, \Delta)$ is a {\em{semi log canonical $n$-fold}} 
(an {\em{slc $n$-fold}}, for short) if $(X^\nu, \Theta)$ is lc. 
We say that $(X, \Delta)$ is a {\em{semi divisorial log terminal 
$n$-fold}} (an {\em{sdlt $n$-fold}}, for short) if $X_i$ is normal, 
that is, $X_i^\nu$ is isomorphic to $X_i$, and $(X^{\nu}, \Theta)$ is dlt. 
\end{defn}

We recall a very important 
example of slc pairs. 

\begin{ex}
Let $(X, \Delta)$ be a $\mathbb Q$-factorial lc pair such that 
$\Delta$ is a $\mathbb Q$-divisor. 
We put $S=\llcorner \Delta\lrcorner$. 
Assume that $(X, \Delta-\varepsilon S)$ is klt for some $0<\varepsilon \ll 1$. 
Then $(S, \Delta_S)$ is slc where $K_S+\Delta_S=(K_X+\Delta)|_S$. 
\end{ex}

\begin{rem}
Let $(X, \Delta)$ be a dlt pair such that 
$\Delta$ is a $\mathbb Q$-divisor. 
We put $S=\llcorner \Delta\lrcorner$. 
Then it is well known that $(S, \Delta_S)$ is 
sdlt where $K_S+\Delta_S=(K_X+\Delta)|_S$. 
\end{rem}

The following theorem was originally proved by 
Christopher Hacon (cf.~\cite[Theorem 10.4]{F-fund}, 
\cite[Theorem 3.1]{kk}). 
For a simpler proof, see \cite[Section 4]{fujino-ss}. 

\begin{thm}[Dlt blow-up]\label{dltblowup}
Let $X$ be a normal quasi-projective variety and let 
$\Delta$ be an effective $\mathbb R$-divisor on $X$ such 
that $K_X+\Delta$ is $\mathbb R$-Cartier. Suppose that $(X,\Delta)$ is lc.
Then there exists a projective birational 
morphism $\varphi:Y\to X$ from a normal quasi-projective 
variety $Y$ with the following properties{\em{:}} 
\begin{itemize}
\item[(i)] $Y$ is $\mathbb Q$-factorial, 
\item[(ii)] $a(E, X, \Delta)= -1$ for every  
$\varphi$-exceptional divisor $E$ on $Y$, and
\item[(iii)] for $$
\Gamma=\varphi^{-1}_*\Delta+\sum _{E: {\text{$\varphi$-exceptional}}}E, 
$$ it holds that  $(Y, \Gamma)$ is dlt and $K_Y+\Gamma=\varphi^*(K_X+\Delta)$.
\end{itemize}
\end{thm}

The above theorem is very useful for the study of log canonical 
singularities (cf. \cite{F-index}, \cite{F-fund}, 
\cite{fujino-iso}, \cite{gongyo-weak}, 
\cite{gongyo-aban}, \cite{kk}, and 
\cite{fujino-gongyo}). 
We will repeatedly use it in the subsequent sections. 

\begin{say}[Log pluricanonical representations] 
Nakamura--Ueno (\cite{nu}) 
and Deligne proved the following theorem 
(see \cite[Theorem 14.10]{u}). 

\begin{thm}[Finiteness of pluricanonical representations]\label{nakamura-ueno}
Let $X$ be a compact complex Moishezon manifold. 
Then the image of the group homomorphism 
$$\rho_{m}:\Bim(X) \to \Aut_{\mathbb{C}}(H^0(X,mK_X))$$
 is finite, where $\Bim(X)$ is the group of bimeromorphic maps from $X$ to itself.
\end{thm}

For considering the logarithmic version of Theorem \ref{nakamura-ueno}, 
we need the notion of {\em{$B$-birational maps}} 
and {\em{$B$-pluricanonical representations}}.

\begin{defn}[{\cite[Definition 3.1]{fujino-abundance}}]\label{B-bir} 
Let $(X,\Delta)$ (resp.~$(Y,\Gamma)$) be a pair such that $X$ (resp.~$Y$) is 
a normal scheme with a $\mathbb{Q}$-divisor $\Delta$ (resp.~$\Gamma$) 
such that $K_X+\Delta$ (resp.~$K_Y+\Gamma$) is $\mathbb{Q}$-Cartier. 
We say that a proper birational map $f:(X,\Delta)\dashrightarrow (Y,\Gamma)$ 
is {\em {$B$-birational}} if there exists a common resolution 
\begin{equation*}
\xymatrix{ & W\ar[dl]_{\alpha} \ar[dr]^{\beta}\\
 X \ar@{-->}[rr]_{f}  & & Y}
\end{equation*}
such that $$\alpha^*(K_X+\Delta)=\beta^*(K_Y+\Gamma). $$
This means that it holds that $E=F$ when we put $K_W=\alpha^*(K_X +\Delta)+E$ 
and $K_W=\beta^*(K_Y+\Gamma)+F$. 

Let $D$ be a $\mathbb Q$-Cartier $\mathbb Q$-divisor on $Y$. 
Then we define 
$$
f^*D:=\alpha_*\beta^*D. 
$$It is easy to see that $f^*D$ is independent of the common resolution $\alpha:W\to X$ and 
$\beta:W\to Y$. 

Finally, we put 
$$\mathrm{Bir}(X,\Delta)=\{\sigma \,|\, \sigma :(X,\Delta) 
\dashrightarrow (X,\Delta) \text{ is $B$-birational} \}. $$ 
It is obvious that $\Bir (X, \Delta)$ has a natural group structure. 
\end{defn}

\begin{rem}
In Definition \ref{B-bir}, let $\psi:X'\to X$ be a proper birational morphism from a 
normal scheme $X'$ such that $K_{X'}+\Delta'=\psi^*(K_X+\Delta)$. 
Then we can easily check that $\Bir (X, \Delta)\simeq \Bir (X', \Delta')$ by 
$g\mapsto \psi^{-1}\circ g\circ \psi$ for $g\in \Bir (X, \Delta)$. 
\end{rem}

We give a basic example of $B$-birational maps. 

\begin{ex}[Quadratic transformation]\label{ex-q} 
Let $X=\mathbb P^2$ and let $\Delta$ be 
the union of three general lines on $\mathbb P^2$. 
Let $\alpha:W\to X$ be the blow-up at the three intersection points of 
$\Delta$ and let $\beta:W\to X$ be the blow-down of the strict transform 
of $\Delta$ on $W$. 
Then we obtain the {\em{quadratic transformation}} $\varphi$. 
\begin{equation*}
\xymatrix{ & W\ar[dl]_{\alpha} \ar[dr]^{\beta}\\
 X \ar@{-->}[rr]_{\varphi}  & & X}
\end{equation*} 
For the details, see \cite[Chapter V Example 4.2.3]{hartshorne}. 
In this situation, it is easy to see that 
$$
\alpha^*(K_X+\Delta)=K_W+\Theta=\beta^*(K_X+\Delta). 
$$ 
Therefore, $\varphi$ is a $B$-birational map of the pair $(X, \Delta)$. 
\end{ex}

\begin{defn}[{\cite[Definition 3.2]{fujino-abundance}}]\label{B-repre} 
Let $X$ be a pure $n$-dimensional normal scheme and let $\Delta$ be 
a $\mathbb{Q}$-divisor, and let $m$ be a nonnegative integer such 
that $m(K_X+\Delta)$ is Cartier. A $B$-birational map 
$\sigma \in \mathrm{Bir}(X,\Delta)$ defines a linear 
automorphism of $H^0(X,m(K_X+\Delta))$. Thus we get the group homomorphism 
$$\rho_{m}:\mathrm{Bir}(X,\Delta) \to \Aut_{\mathbb{C}}(H^0(X,m(K_X+\Delta))).$$
The homomorphism $\rho_{m}$ is called 
a {\em $B$-pluricanonical representation} or 
{\em{log pluricanonical representation}} for $(X,\Delta)$. 
We sometimes simply denote $\rho _m (g)$ by $g^*$ for 
$g\in \Bir (X, \Delta)$ if there is no 
danger of confusion.
\end{defn}
\end{say} 

In Subsection \ref{kltcase}, we will introduce and 
consider {\em{$\widetilde B$-birational maps}} 
and {\em{$\widetilde B$-pluricanonical 
representations}} for subklt pairs (cf.~Definition \ref{tilde-rep}). 
In some sense, they are generalizations of Definitions \ref{B-bir} and \ref{B-repre}. 
We need them for our proof of Theorem \ref{main-thm}. 

\begin{rem}\label{rem215} 
Let $(X, \Delta)$ be a projective 
dlt pair. 
We note that $g\in \Bir (X, \Delta)$ does not 
necessarily induce a birational 
map $g|_T:T\dashrightarrow T$, where 
$T=\llcorner \Delta\lrcorner$ (see Example \ref{ex-q}). 
However, $g\in \Bir (X, \Delta)$ 
induces an automorphism 
$$
g^*: H^0(T, \mathcal O_T(m(K_T+\Delta_T)))\overset{\sim}\longrightarrow
H^0(T, \mathcal O_T(m(K_T+\Delta_T)))
$$
where $(K_X+\Delta)|_T=K_T+\Delta_T$ and 
$m$ is a nonnegative integer such that 
$m(K_X+\Delta)$ is Cartier (see the proof of \cite[Lemma 4.9]{fujino-abundance}). 
More precisely, let  
\begin{equation*}
\xymatrix{ & W\ar[dl]_{\alpha} \ar[dr]^{\beta}\\
 X \ar@{-->}[rr]_{g}  & & X}
\end{equation*}
be a common log resolution such that 
$$\alpha^*(K_X+\Delta)=K_W+\Theta=\beta^*(K_X+\Delta).$$
Then we can easily see that 
$$
\alpha_*\mathcal O_S\simeq \mathcal O_T\simeq \beta_*\mathcal O_S, 
$$ 
where $S=\Theta^{=1}$, by 
the Kawamata--Viehweg vanishing theorem. 
Thus we obtain an automorphism 
\begin{align*}
g^*: H^0(T, \mathcal O_T(m(K_T+\Delta_T)))&\overset{\beta^*}{\longrightarrow} 
H^0(S, \mathcal O_S(m(K_S+\Theta_S)))\\ &\overset{{\alpha^*}^{-1}}\longrightarrow 
H^0(T, \mathcal O_T(m(K_T+\Delta_T)))
\end{align*}
where $(K_W+\Theta)|_S=K_S+\Theta_S$. 
\end{rem}
 
Let us recall an important lemma on $B$-birational maps, which will be 
used in the proof of the main theorem (cf.~Theorem \ref{semi-ample1}). 

\begin{lem}\label{newlem} 
Let $f:(X, \Delta)\to (X', \Delta')$ be a $B$-birational 
map between projective dlt pairs. 
Let $S$ be an lc center of $(X, \Delta)$ such that 
$K_S+\Delta_S=(K_X+\Delta)|_S$. 
We take a suitable common log resolution as in {\em{Definition \ref{B-bir}}}. 
\begin{equation*}
\xymatrix{ & (W, \Gamma)\ar[dl]_{\alpha} \ar[dr]^{\beta}\\
 (X, \Delta) \ar@{-->}[rr]_{f}  & & (X', \Delta')}
\end{equation*}
Then we can find an lc center 
$V$ of $(X, \Delta)$ contained in $S$ with $K_V+\Delta_V=(K_X+\Delta)|_V$, an 
lc center $T$ of $(W, \Gamma)$ with $K_T+\Gamma_T=(K_X+\Delta)|_T$, and 
an lc center $V'$ of $(X', \Delta')$ with $K_{V'}+\Delta'_{V'} 
=(K_{X'}+\Delta')|_{V'}$ such that 
the following conditions hold. 
\begin{itemize}
\item[(a)] $\alpha|_T$ and $\beta|_T$ are 
$B$-birational morphisms. 
\begin{equation*}
\xymatrix{ & (T, \Gamma_T)\ar[dl]_{\alpha|_T} \ar[dr]^{\beta|_T}\\
 (V, \Delta_V)  & & (V', \Delta'_{V'})}
\end{equation*}
Therefore, $(\beta|_T)\circ (\alpha|_T)^{-1}:(V, \Delta_V)\dashrightarrow 
(V', \Delta'|_{V'})$ is a $B$-birational 
map. 
\item[(b)] $H^0(S, m(K_S+\Delta_S))\simeq 
H^0(V, m(K_V+\Delta_V))$ by the natural restriction map where 
$m$ is a nonnegative integer such that 
$m(K_X+\Delta)$ is Cartier. 
\end{itemize}
\end{lem}
\begin{proof}
See Claim ($A_n$) and Claim ($B_n$) in the proof of 
\cite[Lemma 4.9]{fujino-abundance}. 
\end{proof}

\begin{say}[Numerical dimesions]
In Section \ref{sec5}, we will use the notion of {\em{Nakayama's numerical 
Kodaira dimension}} for pseudo-effective 
$\mathbb R$-Cartier $\mathbb R$-divisors on normal projective 
varieties. For the details, see \cite{N} and \cite{brian}. 

\begin{defn}[{Nakayama's numerical Kodaira dimension 
(cf.~\cite[V.~2.5.~Definition]{N}}]
Let $D$ be a pseudo-effective $\mathbb R$-Cartier $\mathbb R$-divisor 
on a normal projective variety $X$ and let $A$ be a Carteir divisor on $X$. 
If $H^0(X, \mathcal O_X(\llcorner mD\lrcorner +A))\ne 0$ for infinitely 
many positive integers $m$, then we put 
$$
\sigma (D; A)=\max \left\{ k\in \mathbb Z_{\geq 0}\, \left|\, {\underset{m\to \infty}{\limsup}}
\frac{\dim H^0(X, \mathcal O_X(\llcorner mD\lrcorner +A))}{m^k}>0 \right.\right\}. 
$$
If $H^0(X, \mathcal O_X(\llcorner mD\lrcorner +A))\ne 0$ only for 
finitely many $m\in \mathbb Z_{\geq 0}$, then 
we put $\sigma (D; A)=-\infty$. 
We define {\em{Nakayama's numerical Kodaira dimension $\kappa _{\sigma}$}} 
by
$$
\kappa _\sigma(X, D)=\max \{\sigma(D; A)\, |\, A \ {\text{is a  Cartier divisor on}} \ X\}. 
$$
If $D$ is a nef $\mathbb R$-Cartier $\mathbb R$-divisor on a normal projective 
variety $X$, then it is well known that $D$ is pseudo-effective and 
$$\kappa _\sigma (X, D)=\nu (X, D)$$
where $\nu(X, D)$ is the {\em{numerical Kodaira dimension}} of $D$. 
\end{defn}
\end{say}

We close this section with a remark on the minimal model 
program with scaling. 
For the details, see \cite{bchm} and \cite{birkar}. 

\begin{say}[Minimal model program with ample 
scaling]\label{27} 
Let $f:X\to Z$ be a projective 
morphism between quasi-projective varieties and 
let $(X, B)$ be a $\mathbb Q$-factorial dlt pair. 
Let $H$ be an effective $f$-ample $\mathbb Q$-divisor on $X$ such 
that $(X, B+H)$ is lc and that $K_X+B+H$ is $f$-nef. 
Under these assumptions, we can run the 
minimal model program on $K_X+B$ with scaling of $H$ over $Z$. 
We call it {\em{the minimal model program with 
ample scaling.}}

Assume that $K_X+B$ is not pseudo-effective 
over $Z$.
We note that the above minimal model program 
always terminates at a Mori fiber space structure over 
$Z$. By this 
observation, the results in \cite[Section 2]{fujino-abundance} 
hold in every dimension. 
Therefore, we will freely use the results in 
\cite[Section 2]{fujino-abundance} for {\em{any}} dimensional 
varieties. 

From now on, we assume that $K_X+B$ is pseudo-effective 
and $\dim X=n$. We further assume that 
the weak non-vanishing conjecture (cf.~Conjecture \ref{nonvan3}) 
for projective $\mathbb Q$-factorial dlt pairs holds in dimension $\leq n$. Then 
the minimal model program on $K_X+B$ with 
scaling of $H$ over $Z$ terminates with a minimal model of $(X, B)$ 
over $Z$ by \cite[Theorems 1.4, 
1.5]{birkar}. 
\end{say}

\section{Finiteness of log pluricanonical representations}\label{sec3} 
In this section, we give a proof of Theorem \ref{main-thm}. 
All the divisors in this section are $\mathbb Q$-divisors. 
We do not use $\mathbb R$-divisors throughout this section. 
We divide the proof into the three steps:~subklt pairs in \ref{kltcase}, 
lc pairs with big log canonical divisor in \ref{general type case}, 
and lc pairs with semi-ample log canonical divisor in \ref{semi-ample case}.  

\subsection{Klt pairs}\label{kltcase}

In this subsection, we prove Theorem \ref{main-thm} for 
klt pairs. More precisely, we prove Theorem \ref{main-thm} for 
{\em{$\widetilde B$-pluricanonical representations}} 
for projective subklt pairs without assuming the semi-ampleness of 
log canonical divisors. This formulation is indispensable 
for the proof of Theorem \ref{main-thm} for 
lc pairs. 

First, let us introduce the notion of 
{\em{$\widetilde B$-pluricanonical representations for subklt pairs}}. 

\begin{defn}[$\widetilde B$-pluricanonical 
representations for subklt pairs]\label{tilde-rep}
Let $(X, \Delta)$ be an $n$-dimensional projective subklt pair such that 
$X$ is smooth and that $\Delta$ has a simple normal crossing support. 
We write $\Delta=\Delta^+-\Delta^-$ where 
$\Delta^+$ and $\Delta^-$ are effective and have no common irreducible components. 
Let $m$ be a positive integer such that $m(K_X+\Delta)$ is Cartier. 
In this subsection, we always see 
$$
\omega\in H^0(X, m(K_X+\Delta))
$$ 
as a meromorphic $m$-ple $n$-form on $X$ which vanishes along 
$m\Delta^-$ and has poles at most $m\Delta^+$. 
By $\Bir (X)$, we mean the group of all the 
birational mappings of $X$ onto itself. 
It has a natural group structure induced by the composition of birational maps. 
We define 
$$
\widetilde \Bir _m(X, \Delta)=\left\{g \in \Bir (X)\, \left|\, 
\begin{array}{l} g^*\omega\in H^0(X, m(K_X+\Delta)) \ \text{for} 
\\{\text{every}}\ 
\omega\in H^0(X, m(K_X+\Delta))\end{array}\right. \right\}. 
$$
Then it is easy to see that 
$\widetilde \Bir _m (X, \Delta)$ is a subgroup of $\Bir (X)$. 
An element $g\in \widetilde \Bir _m (X, \Delta)$ is called 
a {\em{$\widetilde B$-birational map}} 
of $(X, \Delta)$. By the definition of $\widetilde \Bir _m (X, \Delta)$, 
we get the group homomorphism 
$$
\widetilde \rho_m:\widetilde \Bir _m (X, \Delta) 
\to \Aut _{\mathbb C}(H^0(X, m(K_X+\Delta))). 
$$ 
The homomorphism $\widetilde \rho_m$ is called 
the {\em{$\widetilde{B}$-pluricanonical representation}} of 
$\widetilde \Bir _m(X, \Delta)$. 
We sometimes simply denote $\widetilde \rho _m (g)$ by $g^*$ 
for $g\in \widetilde \Bir _m (X, \Delta)$ if there is no danger of confusion. 
There exists a natural inclusion $\Bir (X, \Delta)\subset \widetilde \Bir _m(X, \Delta)$ by 
the definitions. 
\end{defn}

Next, let  us recall the notion of {\em{$L^{2/m}$-integrable $m$-ple $n$-forms}}. 

\begin{defn}\label{integrale-def} 
Let $X$ be an $n$-dimensional connected complex manifold and  
let $\omega$ be a meromorphic $m$-ple $n$-form. Let $\{U_{\alpha}\}$ 
be an open covering of $X$ with holomorphic coordinates 
$$(z_{\alpha}^1,z_{\alpha}^2, \cdots, z_{\alpha}^n).$$
We can write 
\begin{equation*}\omega|_{U_{\alpha}}
=\varphi_{\alpha}(dz_{\alpha}^1\wedge \cdots \wedge dz_{\alpha}^n)^m,
\end{equation*}
where $\varphi_{\alpha}$ is a meromorphic function on $U_{\alpha}$.
We give $(\omega \wedge \bar{\omega})^{1/m}$ by 
\begin{equation*}(\omega \wedge \bar{\omega})^{1/m}|_{U_{\alpha}}
=\left( \frac{\sqrt{-1}}{2\pi} \right)^n |\varphi_{\alpha}|^{2/m}dz_{\alpha}^1
\wedge d\bar{z}_{\alpha}^1 \cdots \wedge dz_{\alpha}^n \wedge d\bar{z}_{\alpha}^n.
\end{equation*}
We say that a meromorphic $m$-ple $n$-form 
$\omega$ is {\em $L^{2/m}$-integrable} 
if $$\int_{X} (\omega \wedge \bar{\omega})^{1/m} < \infty. 
$$
\end{defn}

We can easily check the following two lemmas. 
 
\begin{lem}\label{hol} Let $X$ be a compact connected complex 
manifold and let $D$ be a reduced normal crossing 
divisor on $X$. Set $U =X \setminus D$. If $\omega$ is 
an $L^2$-integrable meromorphic $n$-form such that $\omega|_{U}$ is holomorphic, 
then $\omega$ is a holomorphic $n$-form.
\end{lem}
\begin{proof}See, for example,  \cite[Theorem 2.1]{s} or \cite[Proposition 16]{k3}.
\end{proof}

\begin{lem}[{cf.~\cite[Lemma 4.8]{gongyo-aban}}]\label{integrable} 
Let $(X,\Delta)$ be a projective subklt pair such that $X$ is 
smooth and $\Delta$ has a simple normal crossing support. 
Let $m$ be a positive integer 
such that $m\Delta$ is Cartier 
and let $\omega \in H^0(X, \mathcal{O}_{X}(m(K_X+\Delta)))$ be a 
meromorphic $m$-ple $n$-form. Then $\omega$ is $L^{2/m}$-integrable. 
\end{lem}

By Lemma \ref{integrable}, we obtain the following 
result. 
We note that the proof of \cite[Proposition 4.9]{gongyo-aban} works without any changes 
in our setting. 
 
\begin{prop}\label{algebraic} 
Let $(X,\Delta)$ be an $n$-dimensional projective subklt pair such 
that $X$ is smooth, connected, and $\Delta$ has a simple normal 
crossing support. Let $g \in \widetilde \Bir_m(X,\Delta)$ be a 
$\widetilde B$-birational map where $m$ is a positive integer 
such that $m\Delta$ is Cartier, 
and let $$\omega \in H^0(X, m(K_X+\Delta))$$ be a nonzero meromorphic 
$m$-ple $n$-form on $X$. 
Suppose that $g^*\omega= \lambda \omega$ for some 
$\lambda \in \mathbb{C}$. Then there exists a positive integer 
$N_{m,\omega}$ such that $\lambda^{N_{m,\omega}}=1$ and 
$N_{m,\omega}$ does not depend on $g$. 
\end{prop}

\begin{rem}\label{rem_betti}By the proof 
of \cite[Proposition 4.9]{gongyo-aban} and \cite[Theorem 14.10]{u}, 
we know that $\varphi(N_{m, \omega}) \leq b_n (Y')$, where 
$b_n(Y')$ is the $n$-th Betti number of $Y'$ which 
is in the proof of \cite[Proposition 4.9]{gongyo-aban} and 
$\varphi$ is the Euler function.  
\end{rem}

\begin{prop}[cf. {\cite[Proposition 14.7]{u}}]\label{semi-simpleness} 
Let $(X,\Delta)$ be a projective 
subklt pair such that $X$ is 
smooth, connected, and $\Delta$ has a simple normal crossing support, and let
$$\widetilde {\rho}_{m}:\widetilde \Bir _m(X, \Delta) \to 
\Aut_{\mathbb{C}}(H^0(X,m(K_X+\Delta)))$$ 
be the $\widetilde B$-pluricanonical representation of $\widetilde 
\Bir_m (X, \Delta)$ where $m$ is a positive 
integer such that $m\Delta$ is Cartier. Then $\widetilde {\rho}_{m}(g)$ is 
semi-simple for every $g\in \widetilde \Bir_m (X, \Delta)$. 
\end{prop}

\begin{proof}If $\widetilde {\rho}_{m}(g)$ is not semi-simple, there 
exist two linearly independent elements $\varphi_1,\ \varphi_2 
\in H^0(X,m(K_X+\Delta))$ and nonzero $\alpha \in \mathbb{C}$ such that 
$$g^*\varphi_1=\alpha\varphi_1+\varphi_2,\ g^*\varphi_2=\alpha \varphi_2 
$$ 
by considering Jordan's decomposition of $g^*$. 
Here, we denote $\widetilde {\rho}_m(g)$ by $g^*$ for simplicity. 
By 
Proposition \ref{algebraic}, we see that $\alpha$ is a root of unity. 
Let $l$ be a positive integer. Then we have 
$$(g^l)^*\varphi_1=\alpha^l\varphi _1+l\alpha^{l-1}\varphi_2. 
$$ 
Since $g$ is a birational map, we have $$\int_{X} 
(\varphi_1 \wedge \bar{\varphi}_1)^{1/m}=\int_{X} 
((g^l)^*\varphi_1 \wedge (g^l)^*\bar{\varphi}_1)^{1/m}. $$ On the other hand, we have 
$$
\lim_{l \to \infty} \int_{X} ((g^l)^*\varphi_1 \wedge (g^l)^*\bar{\varphi}_1)^{1/m}
= \infty.
$$ 
For details, see the proof of \cite[Proposition 14.7]{u}. 
However,  we know $\int_{X} (\varphi_1 \wedge \bar{\varphi}_1)^{1/m}< 
\infty $ by Lemma \ref{integrable}. This is a contradiction.
\end{proof}

\begin{prop}\label{algebraic-2} 
The number $N_{m,\omega}$ 
in {\em{Proposition \ref{algebraic}}} 
is uniformly bounded for every $\omega\in H^0(X, m(K_X+\Delta))$. 
Therefore, we can take a positive integer $N_m$ such that 
$N_m$ is divisible by $N_{m, \omega}$ for every $\omega$. 
\end{prop}

\begin{proof}We consider the projective space bundle 
$$\pi:M:=\mathbb{P}_{X}(\mathcal{O}_{X}(-K_X) \oplus \mathcal{O}_X) \to X$$
and 
\begin{align*}
V&:=M\times \mathbb{P}(H^0(X, \mathcal{O}_{X}(m(K_X+\Delta)))) \\ &\to X 
\times   \mathbb{P}(H^0(X, \mathcal{O}_{X}(m(K_X+\Delta)))). 
\end{align*}
We fix a basis $\{\omega_0,\,\omega_1,\,\dots,\,\omega_N\}$ of 
$H^0(X, \mathcal{O}_{X}(m(K_X+\Delta)))$. 
By using this basis, 
we can identify 
$\mathbb{P}(H^0(X, \mathcal{O}_{X}(m(K_X+\Delta))))$ with 
$\mathbb{P}^{N}$. We write  the coordinate of $\mathbb{P}^N$ as 
$(a_0: \cdots: a_N)$ under this identification.
Set $\Delta=\Delta^+-\Delta^-$, where $\Delta^+$ and $\Delta^-$ 
are effective and have no common irreducible components. Let $\{U_{\alpha}\}$ 
be coordinate neighborhoods of $X$ with holomorphic coordinates 
$(z_{\alpha}^1,z_{\alpha}^2, \cdots, z_{\alpha}^n)$. For any $i$, 
we can write $\omega_i$ locally as 
\begin{equation*}
\omega_i|_{U_{\alpha}}=\frac{\varphi_{i,\alpha}}{\delta_{i,\alpha}}
(dz_{\alpha}^1\wedge \cdots \wedge dz_{\alpha}^n)^m,
\end{equation*}
where $\varphi_{i,\alpha}$ and $\delta_{i,\alpha}$ are holomorphic 
with no common factors, and $\frac{\varphi_{i,\alpha}}{\delta_{i,\alpha}}$ 
has poles at most $m\Delta^{+}$. We may assume that $\{U_{\alpha}\}$ 
gives a local trivialization of $M$, that is, 
$M|_{U_{\alpha}} 
:= \pi^{-1}U_{\alpha} \simeq U_{\alpha} 
\times \mathbb{P}^{1}$. We set a coordinate 
$(z_{\alpha}^1,z_{\alpha}^2, \cdots, z_{\alpha}^n, \xi_{\alpha}^0:\xi_{\alpha}^1)$ of 
$U_{\alpha} \times \mathbb{P}^1$ with the homogeneous 
coordinate $(\xi_{\alpha}^0:\xi_{\alpha}^1)$ of $\mathbb{P}^1$. Note that 
$$\frac{\xi_{\alpha}^0}{\xi_{\alpha}^1}= 
k_{\alpha \beta} \frac{\xi_{\beta}^0}{\xi_{\beta}^1}\ 
\text{in}\ M|_{U_{\alpha}\bigcap U_{\beta}},$$ 
where $k_{\alpha \beta}
=\mathrm{det}(\partial z_{\beta}^i/\partial z_{\alpha}^j)_{1\leq i,j \leq n}$. 
Set 
$$Y_{U_{\alpha}}=\{ (\xi_{\alpha}^0)^m\prod_{i=0}^{N}
\delta_{i,\alpha}-(\xi_{\alpha}^1)^m\sum_{i=0}^{N}
\hat{\delta}_{i,\alpha}a_i\varphi_{i,\alpha}=0\} 
\subset U_{\alpha} \times \mathbb{P}^1 \times \mathbb{P}^N,$$
where $\hat{\delta}_{i,\alpha}=\delta_{0,\alpha} 
\cdots \delta_{i-1,\alpha} \cdot \delta_{i+1,\alpha} \cdots \delta_{N,\alpha}$.
By easy calculations, we see that $\{ Y_{U_{\alpha}}\}$ can be 
patched and 
we obtain $Y$. 
We note that $Y$ 
may have singularities and be reducible. The induced 
projection $f:Y \to \mathbb{P}^N$ is surjective and equidimensional. 
Let $q:Y\to X$ be the natural projection. 
By the same arguments as in the proof of 
\cite[Theorem 14.10]{u}, we have a suitable stratification 
$\mathbb P^N=\amalg _i S_i$, where 
$S_i$ is smooth and locally closed in $\mathbb P^N$ for every 
$i$, such that $f^{-1}(S_i)\to S_i$ has 
a simultaneous resolution with good properties for every $i$. 
Therefore, we may assume that there is a positive constant $b$ such that 
for every $p\in \mathbb P^N$ we have 
a resolution $\mu_p:\widetilde {Y}_p\to 
Y_p:=f^{-1}(p)$ with the properties 
that 
$b_n(\widetilde {Y}_p)\leq b$ and 
that $\widetilde \mu_p^*\Delta\cup \Exc(\widetilde \mu_p)$ has a simple normal crossing 
support, where $\widetilde \mu_p:\widetilde Y_p\overset{\mu_p}{\to} Y_p\overset{q}{\to} X$.  
Thus, by Remark \ref{rem_betti}, we 
obtain Proposition \ref{algebraic-2}. 
\end{proof}

Now we have the main theorem of this subsection. 
We will use it in the following subsections. 

\begin{thm}\label{finiteness-klt}Let $(X,\Delta)$ be a 
projective subklt pair such that 
$X$ is smooth, $\Delta$ has a simple normal crossing support, and 
$m(K_X + \Delta)$ is Cartier where $m$ is a positive integer. 
Then $\widetilde {\rho}_{m}
(\widetilde \Bir _m(X, \Delta))$ is a 
finite group. 
\end{thm}
\begin{proof} 
By Proposition 
\ref{semi-simpleness}, we see that $\widetilde {\rho}_m(g)$ is diagonalizable. Moreover,  
Proposition \ref{algebraic-2} implies that the order of $\widetilde {\rho}_m(g)$ is 
bounded by a positive constant $N_m$ which is 
independent of $g$. Thus 
$\widetilde {\rho}_{m}(\widetilde \Bir _m (X, \Delta))$ is a finite 
group by Burnside's theorem 
(see, for example, \cite[Theorem 14.9]{u}). 
\end{proof}

As a corollary, we obtain Theorem \ref{main-thm} 
for klt pairs without assuming the semi-ampleness of 
log canonical divisors. 

\begin{cor}\label{klt-case} 
Let $(X, \Delta)$ be a projective klt pair such that 
$m(K_X+\Delta)$ is Cartier where $m$ is a positive integer. 
Then $\rho _m (\Bir (X, \Delta))$ is a finite group.  
\end{cor}
\begin{proof}
Let $f:Y\to X$ be a log resolution of $(X, \Delta)$ such that 
$K_Y+\Delta_Y=f^*(K_X+\Delta)$. 
Since 
$$\rho_m (\Bir (Y, \Delta_Y))\subset \widetilde \rho _m 
(\widetilde \Bir _m(Y, \Delta_Y)), $$ 
$\rho_m(\Bir (Y, \Delta_Y))$ is a finite group by 
Theorem \ref{finiteness-klt}.  
Therefore, we obtain that 
$\rho_m(\Bir (X, \Delta))\simeq \rho_m(\Bir (Y, \Delta_Y))$ is a finite 
group. 
\end{proof}

\subsection{Lc pairs with big log canonical divisor}\label{general type case}
In this subsection, we prove the following theorem. 
The proof is essentially the same as that of Case 1 in 
\cite[Theorem 3.5]{fujino-abundance}. 

\begin{thm}\label{finiteness1}Let $(X,\Delta)$ be a 
projective sublc 
pair such that $K_X + \Delta$ is big. 
Let $m$ be a positive integer such that $m(K_X+\Delta)$ is Cartier. 
Then $\rho_{m}(\mathrm{Bir}(X,\Delta))$ is a finite group.  
\end{thm}

Before we start the proof of Theorem \ref{finiteness1}, 
we give a remark. 

\begin{rem}
By Theorem \ref{finiteness1}, 
when $K_X+\Delta$ is big,  
Theorem \ref{main-thm}, the main theorem of this paper, 
holds true without assuming that $K_X+\Delta$ 
is semi-ample. 
Therefore, we state Theorem \ref{finiteness1} separately for some future usage 
(cf.~Corollary \ref{finiteness birational auto}). 
In Case \ref{non-dom} in the proof of 
Theorem \ref{semi-ample1}, which is nothing but Theorem \ref{main-thm}, 
we will use the arguments in the proof of 
Theorem \ref{finiteness1}.    
\end{rem}

\begin{proof} 
By taking a log resolution, we can assume that 
$X$ is smooth and $\Delta$ has a simple normal crossing 
support. 
By Theorem \ref{finiteness-klt}, we can 
also assume that 
$\Delta^{=1} \not= 0$. Since $K_X +\Delta$ is big, for 
a sufficiently large and divisible positive integer $m'$, we obtain an effective 
Cartier divisor $D_{m'}$ such that 
$$m'(K_X+\Delta) \sim_{\mathbb{Z}} \Delta^{=1} +D_{m'}
$$
by Kodaira's lemma. 
It is easy to see that $\Supp g^*\Delta^{=1}\supset\Supp \Delta^{=1}$ for 
every $g\in \Bir (X, \Delta)$. 
This implies that $g^*\Delta^{=1}\geq \Delta^{=1}$. 
Thus, we have a natural inclusion 
$$
\Bir (X, \Delta)\subset \widetilde \Bir_{m'} \left(X, \Delta-\frac{1}{m'}\Delta^{=1}\right).
$$
We consider the $\widetilde B$-birational representation 
\begin{align*}
\widetilde \rho_{m'}: 
\widetilde \Bir_{m'} \left(X, \Delta-\frac{1}{m'}\Delta^{=1}\right) 
\to \Aut_{\mathbb C}H^0(X, 
m'(K_X+\Delta)-\Delta^{=1}). 
\end{align*}
Then, by Theorem \ref{finiteness-klt}, 
$$\widetilde {\rho}_{m'}\left(\widetilde \Bir_{m'} \left(X, \Delta-\frac{1}{m'}\Delta^{=1}
\right)\right)$$ 
is a finite group. 
Therefore, $\widetilde \rho_{m'}(\Bir (X, \Delta))$ is also a finite group. 
We put $a=|\widetilde {\rho}_{m'}(\Bir (X, \Delta))|<\infty$. 
In this situation, we can find a $\Bir(X, \Delta)$-invariant non-zero 
section $s\in H^0(X, a(m'(K_X+\Delta)-\Delta^{=1}))$. 
By using $s$, we have a natural inclusion 
\begin{equation}\tag{$\spadesuit$}\label{siki} 
H^0(X, m(K_X+\Delta))\subseteq H^0(X, (m+m'a)(K_X+\Delta)-a\Delta^{=1}). 
\end{equation} 
By the construction, $\Bir (X, \Delta)$ acts on the both vector spaces 
compatibly. 
We consider the $\widetilde B$-pluricanonical 
representation 
\begin{align*}
\widetilde {\rho}_{m+m'a}&: 
\widetilde \Bir_{m+m'a} \left(X, \Delta
-\frac{a}{m+m'a}\Delta^{=1}\right)
\\ &\to \Aut _{\mathbb C} H^0(X, (m+m'a)(K_X+\Delta)-a\Delta^{=1}).
\end{align*}
Since $$\left(X, \Delta-\frac{a}{m+m'a}\Delta^{=1}\right)$$ is subklt, we have 
that
$$
\widetilde \rho _{m+m'a}\left(\widetilde \Bir _{m+m'a}
\left(X, \Delta-\frac{a}{m+m'a}\Delta^{=1}\right)\right)
$$ 
is a finite group by Theorem \ref{finiteness-klt}. Therefore, 
$\widetilde \rho _{m+m'a}(\Bir (X, \Delta))$ is also a finite group. 
Thus, we obtain that $\rho _m(\Bir (X, \Delta))$ is a finite group 
by the $\Bir (X, \Delta)$-equivariant embedding (\ref{siki}). 
\end{proof}

The following corollary is an answer to the question raised by 
Cacciola and Tasin. 
It is a generalization of the well-known 
finiteness of birational automorphisms of varieties of general type 
(cf.~\cite[Corollary 14.3]{u}). 

\begin{cor}\label{finiteness birational auto}
Let $(X,\Delta)$ be a projective sublc 
pair such that $K_X+\Delta$ is big.
Then $\Bir (X, \Delta)$ is a finite group.
\end{cor}

\begin{proof} 
We consider the rational map 
$$
\Phi_m:=\Phi _{|m(K_X+\Delta)|}: X\dashrightarrow \mathbb P(H^0(X, m(K_X+\Delta)))
$$ 
associated to the complete linear system $|m(K_X+\Delta)|$, where 
$m$ is a positive integer 
such that 
$m(K_X+\Delta)$ is Cartier. 
By taking $m\gg 0$, we may assume that 
$\Phi _m:X\dashrightarrow V$ is birational because 
$K_X+\Delta$ is big,  
where $V$ is the image of $X$ by $\Phi _m$. 
The log pluricanonical representation 
$$
\rho_m:\Bir (X, \Delta)\to \Aut _{\mathbb C}(H^0(X, m(K_X+\Delta)))
$$ 
induces the group homomorphism 
$$
\bar\rho_m:\Bir (X, \Delta)\to \Aut (\mathbb P(H^0(X, m(K_X+\Delta)))). 
$$
Note that $\bar\rho_m(g)$ leaves $V$ invariant for every $g\in \Bir (X, \Delta)$ by the 
construction. 
Since $\Phi _m:X\dashrightarrow V$ is birational, $\bar\rho_m$ is injective. 
On the other hand, we see that 
$\bar\rho_m(\Bir (X, \Delta))$ is finite by Theorem \ref{finiteness1}. 
Therefore, $\Bir (X, \Delta)$ is a finite group. 
\end{proof}

\begin{rem} 
By the proof of Corollary \ref{finiteness birational auto} and Theorem \ref{finiteness-klt}, 
we obtain the following finiteness of $\widetilde \Bir _m(X, \Delta)$. 

Let $(X, \Delta)$ be a projective subklt pair such that 
$X$ is smooth and that $\Delta$ has a simple normal crossing support. 
Let $m$ be a positive integer such that $m(K_X+\Delta)$ is Cartier and 
that $|m(K_X+\Delta)|$ defines a birational map. 
Then $\widetilde \Bir _m (X, \Delta)$ is a finite group. 
\end{rem}

\subsection{Lc pairs with semi-ample log canonical divisor}
\label{semi-ample case}
Theorem \ref{semi-ample1} 
is one of the main results of this paper (see Theorem \ref{main-thm}). 
We will treat many applications of Theorem \ref{semi-ample1} 
in Section \ref{sec1}. 

\begin{thm}\label{semi-ample1}Let $(X,\Delta)$ be an $n$-dimensional 
projective lc pair such that $K_X+\Delta$ is semi-ample. 
Let $m$ be a positive integer such that 
$m(K_X+\Delta)$ is Cartier. 
Then $\rho_{m}(\mathrm{Bir}(X,\Delta))$ is a finite group. 
\end{thm}

\begin{proof}We show the statement by the induction on $n$. 
By taking a dlt blow-up (cf.~Theorem \ref{dltblowup}), 
we may assume that $(X,\Delta)$ is a $\mathbb{Q}$-factorial dlt pair. 
Let $f:X \to Y$ be a projective surjective morphism associated to 
$k(K_X+\Delta)$ for a sufficiently large and divisible positive integer $k$. 
By Corollary \ref{klt-case}, we may assume that $\llcorner \Delta \lrcorner\not = 0$.

\begin{case}\label{dominant_case}$\llcorner \Delta^h \lrcorner \not =0$, 
where $\Delta^h$ is the horizontal part of $\Delta$ 
with respect to $f$.
\end{case}

In this case, we put 
$T = \llcorner \Delta \lrcorner$. Since $m(K_X+\Delta)\sim_{\mathbb Q, Y}0$, we see that 
$$H^0(X, \mathcal{O}_{X}(m(K_X+\Delta)-T)) =0.$$
Thus the restricted map
$$H^0(X, \mathcal{O}_{X}(m(K_X+\Delta))) \to H^0(T, \mathcal{O}_{T}(m(K_T+\Delta_T))) 
$$ 
is injective,  where $K_T+\Delta_T=(K_X+\Delta)|_T$. 
Let 
$(V_i, \Delta_{V_i})$ 
be the disjoint union of all the $i$-dimensional lc centers of $(X, \Delta)$ for 
$0\leq i\leq n-1$. We note that 
$\rho_m(\Bir (V_i, \Delta_{V_i}))$ is a finite group 
for every $i$ by the induction on dimension. 
We put $k_i=|\rho_m(\Bir (V_i, \Delta_{V_i}))|<\infty$ for 
$0\leq i\leq n-1$. 
Let $l$ be the least common multiple of $k_i$ for $0\leq i\leq n-1$. 
Let $T=\cup _j T_j$ be the irreducible decomposition. 
Let $g$ be an element of $\Bir (X, \Delta)$. 
By repeatedly using Lemma \ref{newlem}, for 
every $T_j$, we can find lc centers $S_j^i$ of $(X, \Delta)$
$$
\begin{matrix}
X&\overset{g}{\dashrightarrow} &X& \overset{g}{\dashrightarrow} 
&X& \overset{g}{\dashrightarrow}&\cdots&\overset{g}{\dashrightarrow}& X&
\dashrightarrow\\
\cup & & \cup & & \cup &&&&\cup &\\ 
S_j^0&& S_j^1&& S_j^2&&&&S_j^k&
\end{matrix}
$$
such that $S_j^0\subset T_j$, $S_j^i\dashrightarrow S_j^{i+1}$ is a $B$-birational map 
for every $i$, and $$H^0(T_j, m(K_{T_j}+\Delta_{T_j}))\simeq 
H^0(S_j^0, m(K_{S_j^0}+\Delta_{S_j^0}))$$ by the natural 
restriction map, where 
$K_{T_j}+\Delta_{T_j}=(K_X+\Delta)|_{T_j}$ and 
$K_{S_j^0}+\Delta_{S_j^0}=(K_X+\Delta)|_{S_j^0}$. 
Since there are only finitely many lc centers of $(X, \Delta)$, 
we can find $p_j<q_j$ such that 
$S_j^{p_j}=S_j^{q_j}$ and that 
$S_j^{p_j}\ne S_j ^r$ for $r=p_j+1, \cdots, q_j-1$. 
Therefore, $g$ induces a $B$-birational map 
$$
\widetilde g: \coprod_{p_j\leq r\leq q_j-1}S_j^r\dashrightarrow 
\coprod _{p_j\leq r\leq q_j-1}S_j^r 
$$
for every $j$. 
We have an embedding 
$$
H^0(T, \mathcal O_T(m(K_T+\Delta_T)))\subset \bigoplus _j H^0(S_j^{p_j}, m(K_{S_j^{p_j}}
+\Delta_{S_j^{p_j}})), 
$$
where $K_{S_j^{p_j}}+\Delta_{S_j^{p_j}}=(K_X+\Delta)|_{S_j^{p_j}}$ for every $j$. 
First, by the following commutative diagram (cf.~Remark \ref{rem215})  
\[\xymatrix{
0\ar[r]&H^0(T, \mathcal O_T(m(K_T+\Delta_T))) \ar[d]_{(g^*)^l} \ar[r] 
& \bigoplus _j H^0(S_j^{p_j}, m(K_{S_j^{p_j}}+\Delta_{S_j^{p_j}}))) 
\ar[d]^{(\widetilde {g}^*)^l=\id} \\
0\ar[r]& H^0(T, \mathcal O_T(m(K_T+\Delta_T)))
\ar[r] &  \bigoplus _j H^0(S_j^{p_j}, m(K_{S_j^{p_j}}+\Delta_{S_j^{p_j}}))), \\
}\] 
we obtain $(g^*)^l=\id$ on $H^0(T, m(K_T+\Delta_T))$. 
Next, by the following commutative diagram 
(cf.~Remark \ref{rem215})  
\[\xymatrix{
0\ar[r]&H^0(X, \mathcal O_X(m(K_X+\Delta))) \ar[d]_{(g^*)^l} \ar[r] 
& H^0(T, \mathcal O_T(m(K_T+\Delta_T))) 
\ar[d]^{(g^*)^l=\id} \\
0\ar[r]&H^0(X, \mathcal O_X(m(K_X+\Delta))) 
\ar[r] &   H^0(T, \mathcal O_T(m(K_T+\Delta_T))), \\
}\] 
we have that $(g^*)^l=\id$ on $H^0(X, \mathcal O_X(m(K_X+\Delta)))$. 
Thus we obtain that $\rho_m(\Bir (X, \Delta))$ is a finite group 
by Burnside's theorem (cf.~\cite[Theorem 14.9]{u}). 

\begin{case}\label{non-dom}$\llcorner \Delta^h \lrcorner=0.$ 
\end{case}
We can construct the commutative diagram 
\[\xymatrix{
X' \ar[d]_{f'} \ar[r]^{\varphi} & X \ar[d]^{f} \\
Y'\ar[r]_{\psi} & Y  \\
}\] 
with the following properties: 
\begin{itemize}
\item[(a)] $\varphi:X'\to X$ is a log resolution of $(X, \Delta)$. 
\item[(b)] $\psi:Y'\to Y$ is a resolution of $Y$. 
\item[(c)] there is a simple normal crossing divisor 
$\Sigma$ on $Y'$ such that 
$f'$ is smooth and $\Supp \varphi_*^{-1}\Delta\cup \Exc (\varphi)$ is relatively 
normal crossing over $Y'\setminus \Sigma$. 
\item[(d)] $\Supp f'^*\Sigma$ and $\Supp f'^*\Sigma\cup 
\Exc (\varphi)\cup \Supp \varphi_*^{-1}\Delta$ are simple normal crossing 
divisors on $X'$. 
\end{itemize}
Then we have 
$$K_{X'}+\Delta_{X'}= f'^*(K_{Y'} + \Delta_{Y'} +M),$$
where $K_{X'}+\Delta_{X'}=\varphi^*(K_X+\Delta)$, 
$\Delta_{Y'}$ is the discriminant divisor and $M$ is the moduli part of 
$f':(X', \Delta_{X'})\to Y'$. Note that
$$\Delta_{Y'}=\sum(1-c_Q)Q,
$$
where $Q$ runs through all the prime divisors on $Y'$ and 
$$
c_Q=\sup \{t\in \mathbb Q\, | \, K_{X'}+\Delta_{X'}+tf'^*Q \ \text{is sublc 
over the generic point of} \ Q\}.
$$
When we construct $f':X'\to Y'$, 
we first make $f:X\to Y$ toroidal by 
\cite[Theorem 2.1]{ak}, next make it equi-dimensional by \cite[Proposition 4.4]{ak}, 
and finally obtain $f':X'\to Y'$ by taking a resolution. By the above construction, 
we can assume that $f':X'\to Y'$ factors as 
$$
f':X'\overset{\alpha}{\longrightarrow} \widetilde X\overset{\widetilde f}{\longrightarrow}  Y'
$$
with the following properties. 
\begin{itemize}
\item[(1)] $\widetilde f: (\widetilde U\subset \widetilde X)\to (U_{Y'}\subset Y')$, where 
$U_{Y'}=Y'\setminus \Sigma$ and $\widetilde U$ is some Zariski open 
set of $\widetilde X$, is toroidal and equi-dimensional. 
\item[(2)] $\alpha$ is a projective biratioinal morphism and is an isomorphism over $U_{Y'}$. 
\item[(3)] $\widetilde \varphi:=\varphi\circ \alpha^{-1}:\widetilde X\to X$ is a 
morphism such that $K_{\widetilde X}+\Delta_{\widetilde X}=\widetilde \varphi^*(K_X+\Delta)$ and 
that $\Supp \Delta_{\widetilde X}\subset \widetilde X\setminus \widetilde U$. 
\item[(4)] $\widetilde X$ has only quotient singularities (cf.~\cite[Remark 4.5]{ak}). 
\end{itemize}
For the details, see the arguments in \cite{ak}. 
In this setting, it is easy to see that 
$\Supp \Delta_{X'}^{=1}\subset 
\Supp f'^*\Delta_{Y'}^{=1}$.  
Therefore, $\Delta_{X'}^{=1}\leq f'^*\Delta_{Y'}^{=1}$. 
We can check that every $g \in \Bir (X', \Delta_{X'})=\Bir (X, \Delta)$ induces 
$g_{Y'}\in \Bir (Y', \Delta_{Y'})$ which satisfies the 
following commutative 
diagram (see \cite[Theorem 0.2]{ambro} for the subklt case, 
and \cite[Proposition 8.4.9 (3)]{kollar} for the sublc case). 
\[\xymatrix{
X' \ar[d]_{f'} \ar@{-->}[r]^g \ar@{}[dr]|\circlearrowleft & X' \ar[d]^{f'} \\
Y'\ar@{-->}[r]_{g_{Y'}} & Y' \\
}\]
Therefore, we have 
$\Supp  g_{Y'}^*\Delta_{Y'}^{=1}\supset\Supp  \Delta_{Y'}^{=1}$. 
This implies that $$g_{Y'}^*\Delta_{Y'}^{=1}\geq \Delta_{Y'}^{=1}.$$
Thus there is an effective Cartier divisor $E_g$ on $X'$ such that 
$$
g^*f'^*\Delta_{Y'}^{=1}+E_g\geq f'^*\Delta_{Y'}^{=1}
$$ 
and that the codimension of $f'(E_g)$ in $Y'$ is $\geq 2$. 
We note the definitions of $g^*$ and $g_{Y'}^*$ (cf.~Definition \ref{B-bir}). 
Therefore, $g\in \Bir (X', \Delta_{X'})$ induces an 
automorphism $g^*$ of 
$H^0(X', m'(K_{X'}+\Delta_{X'})-f'^*\Delta_{Y'}^{=1})$ where 
$m'$ is a sufficiently large and divisible positive integer $m'$. 
It is because 
\begin{align*}
&H^0(X', m'(K_{X'}+\Delta_{X'})-g^*f'^*\Delta_{Y'}^{=1})\\
&\subset H^0(X', m'(K_{X'}+\Delta_{X'})-f'^*\Delta_{Y'}^{=1}+E_g)\\
&\simeq H^0(X', m'(K_{X'}+\Delta_{X'})-f'^*\Delta_{Y'}^{=1}). 
\end{align*} 
Here, we used the facts that $m'(K_{X'}+\Delta_{X'})=f'^*(m'(K_{Y'}+\Delta_{Y'}+M))$ and 
that $f'_*\mathcal O_{X'}(E_g)\simeq 
\mathcal O_{Y'}$. 
Thus we have a natural inclusion 
$$
\Bir (X', \Delta_{X'})\subset \widetilde \Bir_{m'} \left(X', \Delta_{X'}-\frac{1}{m'}
f'^*\Delta_{Y'}^{=1}\right). 
$$
Note that $$\left(X', \Delta_{X'}-\frac{1}{m'}f'^*\Delta_{Y'}^{=1}\right)$$ 
is subklt because $\Delta_{X'}^{=1}\leq 
f'^*\Delta_{Y'}^{=1}$. 
Since $K_{Y'} + \Delta_{Y'} +M$ is (nef and) big, for a sufficiently large 
and divisible positive integer $m'$, we obtain an effective Cartier 
divisor $D_{m'}$ such that $$m'(K_{Y'}+\Delta_{Y'}+M) \sim_{\mathbb{Z}}  
\Delta_{Y'} ^{=1} +D_{m'}. $$ 
This means that 
$$
H^0(X', m'(K_{X'}+\Delta_{X'})-f'^*\Delta_{Y'}^{=1})\ne 0. 
$$
By considering the natural inclusion 
$$
\Bir (X', \Delta_{X'})\subset \widetilde \Bir_{m'} \left(X', \Delta_{X'}-\frac{1}{m'}
f'^*\Delta_{Y'}^{=1}\right), 
$$ 
we can use the same arguments as in the proof of Theorem \ref{finiteness1}. 
Thus we obtain the finiteness of $B$-pluricanonical 
representations. 
\end{proof}

\begin{rem}\label{rem310} 
Although we did not explicitly state it,  
in Theorem \ref{finiteness-klt}, 
we do not have to assume that 
$X$ is connected. 
Similarly, we can prove Theorems \ref{finiteness1} and \ref{semi-ample1} 
without assuming that $X$ is connected. 
For the details, see \cite[Remark 4.4]{gongyo-aban}. 
\end{rem}

{We close this section with comments on \cite[Section 3]{fujino-abundance} 
and \cite[Theorem B]{gongyo-aban}. 
In \cite[Section 3]{fujino-abundance}, 
we proved Theorem \ref{semi-ample1} for surfaces. 
There, we do not need the notion of $\widetilde B$-birational 
maps. It is mainly because $Y'$ in Case \ref{non-dom} 
in the proof of Theorem \ref{semi-ample1} is a curve 
if $(X, \Delta)$ is not klt and $K_X+\Delta$ is not big. 
Thus, $g_{Y'}$ is an automorphism of $Y'$. 
In \cite[Theorem B]{gongyo-aban}, we proved 
Theorem \ref{semi-ample1} under the assumption that $K_X+\Delta\sim_{\mathbb Q}0$. In 
that case, Case \ref{dominant_case} 
in the proof of Theorem \ref{semi-ample1} 
is sufficient. Therefore, we do not need the notion of $\widetilde B$-birational 
maps in \cite{gongyo-aban}.  
 
\section{On abundance conjecture for 
log canonical pairs}\label{sec1}

In this section, we treat various applications 
of Theorem \ref{main-thm} on the abundance conjecture 
for (semi) lc pairs (cf.~Conjecture \ref{abun}). 
We note that we only treat $\mathbb Q$-divisors in this section. 

Let us introduce the notion of {\em{nef and log abundant 
$\mathbb Q$-divisors}}. 

\begin{defn}[Nef and log abundant divisors]
Let $(X, \Delta)$ be a sublc pair. 
A closed subvariety $W$ of $X$ is called 
an {\em{lc center}} if there exist a resolution $f:Y\to X$ and 
a divisor $E$ on $Y$ such that 
$a(E, X, \Delta)=-1$ and $f(E)=W$.  
A $\mathbb Q$-Cartier $\mathbb Q$-divisor 
$D$ on $X$ is 
called {\em{nef and log abundant with respect to 
$(X, \Delta)$}} 
if and only if 
$D$ is nef and abundant, and $\nu_W^*D|_W$ is nef and 
abundant for every lc center $W$ of the 
pair $(X, \Delta)$, where 
$\nu_W:W^\nu\to W$ is 
the normalization. 
Let $\pi:X\to S$ be a proper morphism onto a variety $S$. 
Then $D$ is {\em{$\pi$-nef and $\pi$-log abundant with respect 
to $(X, \Delta)$}} if and 
only if $D$ is $\pi$-nef and $\pi$-abundant and 
$(\nu_W^*D|_W)|_{W^{\nu}_\eta}$ is abundant, where 
$W^{\nu}_\eta$ is the generic fiber of $W^{\nu}\to \pi(W)$. 
We sometimes simply say that $D$ is nef and log abundant over 
$S$.  
\end{defn}

The following theorem is one of the main theorems of this section 
(cf.~\cite[Theorem 0.1]{fujino-reid}, \cite[Theorem 4.4]{fujino-book}). 
For a relative version of Theorem \ref{thm1}, see Theorem \ref{thm11} below. 
See also Subsection \ref{subsec41}. 

\begin{thm}\label{thm1}
Let $(X, \Delta)$ be a projective lc pair. 
Assume that $K_X+\Delta$ is nef and 
log abundant. 
Then $K_X+\Delta$ is semi-ample. 
\end{thm}
\begin{proof}
By replacing $(X, \Delta)$ with its dlt blow-up 
(cf.~Theorem \ref{dltblowup}), we can assume that 
$(X, \Delta)$ is dlt and that $K_X+\Delta$ is nef and log abundant. 
We put $S=\llcorner \Delta\lrcorner$. 
Then $(S, \Delta_S)$, where $K_S+\Delta_S=(K_X+\Delta)|_S$, is an sdlt $(n-1)$-fold and 
$K_S+\Delta_S$ is semi-ample by the induction on dimension and Proposition \ref{prop3} below. 
By applying Fukuda's theorem (cf.~\cite[Theorem 1.1]{fujino-bpf}), we obtain that 
$K_X+\Delta$ is semi-ample. 
\end{proof}

We note that Proposition \ref{prop3} is a key result in this paper. 
It heavily depends on Theorem \ref{main-thm}. 

\begin{prop}\label{prop3}
Let $(X, \Delta)$ be a projective slc pair. 
Let $\nu:X^\nu\to X$ be the normalization. 
Assume that $K_{X^\nu}+\Theta=\nu^*(K_X+\Delta)$ is semi-ample. 
Then $K_X+\Delta$ is semi-ample. 
\end{prop}

\begin{proof}
The arguments in \cite[Section 4]{fujino-abundance} 
work by Theorem \ref{main-thm}. 
As we pointed out in \ref{27}, we can freely 
use the results in \cite[Section 2]{fujino-abundance}. 
The finiteness of $B$-pluricanonical 
representations, which was only proved in dimension $\leq 2$ in 
\cite[Section 3]{fujino-abundance}, is now Theorem \ref{main-thm}. 
Therefore, the results in \cite[Section 4]{fujino-abundance} 
hold in any dimension. 
\end{proof}

By combining Proposition \ref{prop3} with 
Theorem \ref{thm1}, we obtain an obvious 
corollary (see also Corollary \ref{cor412}, 
Theorem \ref{thm415}, and Remark \ref{rem416}). 

\begin{cor}\label{cor44} 
Let $(X, \Delta)$ be a projective slc pair and 
let $\nu:X^\nu\to X$ be 
the normalization. 
If $K_{X^\nu}+\Theta=\nu^*(K_X+\Delta)$ is nef and 
log abundant, then $K_X+\Delta$ is semi-ample. 
\end{cor}

We give one more corollary of Proposition \ref{prop3}. 
\begin{cor}\label{45cor}
Let $(X, \Delta)$ be a projective 
slc pair such that 
$K_X+\Delta$ is nef. 
Let $\nu:X^{\nu}\to X$ be the 
normalization. Assume that 
$X^\nu$ is a toric variety. 
Then $K_X+\Delta$ is semi-ample. 
\end{cor}

\begin{proof}
It is well known that every nef $\mathbb Q$-Cartier $\mathbb Q$-divisor 
on a projective toric variety is semi-ample. 
Therefore, this corollary is obvious by Proposition \ref{prop3}. 
\end{proof}
\begin{thm}\label{thm2}
Let $(X, \Delta)$ be a projective $n$-dimensional 
lc pair. 
Assume that 
the abundance conjecture holds for 
projective dlt pairs in dimension $\leq n-1$. 
Then $K_X+\Delta$ is semi-ample if and only if 
$K_X+\Delta$ is nef and abundant. 
\end{thm}
\begin{proof}
It is obvious that $K_X+\Delta$ is nef and abundant 
if $K_X+\Delta$ is semi-ample. 
So, we show that 
$K_X+\Delta$ is semi-ample 
under the assumption that $K_X+\Delta$ is nef and 
abundant. 
By taking a dlt blow-up 
(cf.~Theorem \ref{dltblowup}),  
we can assume that $(X, \Delta)$ is dlt. 
By the assumption, it is easy to see 
that $K_X+\Delta$ is nef and 
log abundant. 
Therefore, by Theorem \ref{thm1}, 
we obtain that $K_X+\Delta$ is semi-ample. 
\end{proof}

The following theorem is an easy consequence of 
the arguments in \cite[Section 7]{kemamc} and 
Proposition \ref{prop3} by the induction on dimension. 
We will treat related topics in Section \ref{sec5} more 
systematically.  

\begin{thm}\label{thm4}
Let $(X, \Delta)$ be a projective $\mathbb Q$-factorial 
dlt $n$-fold such 
that $K_X+\Delta$ is nef. 
Assume that the abundance conjecture 
for projective $\mathbb Q$-factorial 
klt pairs in dimension $\leq n$. 
We further assume that 
the minimal model program with ample scaling terminates 
for projective $\mathbb Q$-factorial klt pairs in dimension $\leq n$. 
Then $K_X+\Delta$ is semi-ample. 
\end{thm}

\begin{proof}
This follows from the arguments in \cite[Section 7]{kemamc} 
by using the minimal model program with ample scaling 
with the aid of Proposition \ref{prop3}. 
Let $H$ be a general effective sufficiently ample Cartier divisor 
on $X$. We run the minimal model program on 
$K_X+\Delta-\varepsilon \llcorner \Delta 
\lrcorner$ with scaling of $H$. 
We note that $K_X+\Delta$ is 
numerically trivial 
on the extremal ray in each step of the above minimal model program 
if $\varepsilon$ is sufficiently small by \cite[Proposition 3.2]{birkar}. 
We also note that, by the induction on dimension, 
$(K_X+\Delta)|_{\llcorner \Delta\lrcorner}$ is semi-ample. 
For the details, see \cite[Section 7]{kemamc}. 
\end{proof}

\begin{rem}
In the proof of Theorem \ref{thm4}, 
the abundance theorem and the 
termination of the minimal model program with 
ample scaling for projective 
$\mathbb Q$-factorial klt pairs in dimension $\leq n-1$ are 
sufficient if $K_X+\Delta-\varepsilon \llcorner \Delta\lrcorner$ is 
not pseudo-effective 
for every $0<\varepsilon \ll 1$ by \cite{bchm} (cf.~\ref{27}). 
\end{rem}

The next theorem is an answer to 
J\'anos Koll\'ar's question 
for {\em{projective}} varieties. He was mainly 
interested in the case where $f$ is birational. 

\begin{thm}\label{thm10}
Let $f:X\to Y$ be a projective morphism between projective varieties. 
Let $(X, \Delta)$ be an lc pair such that $K_X+\Delta$ is 
numerically trivial over $Y$. 
Then $K_X+\Delta$ is $f$-semi-ample. 
\end{thm}

\begin{proof}
By replacing $(X, \Delta)$ with 
its dlt blow-up 
(cf.~Theorem \ref{dltblowup}), 
we can assume that $(X, \Delta)$ is a $\mathbb Q$-factorial 
dlt pair. 
Let $S=\llcorner \Delta\lrcorner =\cup S_i$ be the irreducible decomposition. 
If $S=0$, then $K_X+\Delta$ is $f$-semi-ample by Kawamata's theorem 
(see \cite[Theorem 1.1]{fujino-kawamata}). 
It is because $(K_{X}+\Delta)|_{X_\eta}\sim _\mathbb Q 0$, 
where $X_\eta$ is the generic fiber of $f$,  by 
Nakayama's abundance theorem for klt pairs with 
numerical trivial log canonical divisor (cf.~\cite[Chapter V. 4.9.~Corollary]{N}).  
By the induction on dimension, 
we can assume 
that $(K_X+\Delta)|_{S_i}$ is semi-ample over $Y$ for 
every $i$.  Let $H$ be a general effective sufficiently ample $\mathbb Q$-Cartier 
$\mathbb Q$-divisor 
on $Y$ such that $\llcorner H\lrcorner=0$. 
Then $(X, \Delta+f^*H)$ is dlt, $(K_X+\Delta+f^*H)|_{S_i}$ is 
semi-ample for every $i$.  
By Proposition \ref{prop3}, $(K_X+\Delta+f^*H)|_S$ is 
semi-ample. By applying \cite[Theorem 1.1]{fujino-bpf}, 
we obtain that $K_X+\Delta+f^*H$ is 
$f$-semi-ample. 
We note that $(K_X+\Delta+f^*H)|_{X_\eta}\sim _{\mathbb Q}0$ 
(see, for example, \cite[Theorem 1.2]{gongyo-aban}). 
Therefore, $K_X+\Delta$ is $f$-semi-ample. 
\end{proof}

\begin{rem} 
In Theorem \ref{thm10}, if $\Delta$ is 
an $\mathbb R$-divisor, then we also obtain 
that $K_X+\Delta$ is semi-ample over $Y$ by the 
same arguments as in \cite[Lemma 6.2]{gongyo-minimal} and \cite[Theorem 3.1]{fujino-gongyo}. 
\end{rem}

As a corollary, we obtain a relative version of the main theorem 
of \cite{gongyo-aban}. 

\begin{cor}[{cf.~\cite[Theorem 1.2]{gongyo-aban}}]\label{ko-slc} 
Let $f:X\to Y$ be a projective morphism 
from a projective slc pair $(X, \Delta)$ to a {\em{(}}not necessarily 
irreducible{\em{)}} projective variety $Y$. 
Assume that $K_X+\Delta$ is numerically trivial over $Y$. 
Then $K_X+\Delta$ is $f$-semi-ample. 
\end{cor}

\begin{proof} 
Let $\nu:X^\nu\to X$ be the normalization such that 
$K_{X^\nu}+\Theta=\nu^*(K_X+\Delta)$. By Theorem \ref{thm10}, 
$K_{X^\nu}+\Theta$ is semi-ample over $Y$. 
Let $H$ be a general sufficiently ample $\mathbb Q$-divisor on $Y$ such that 
$K_{X^\nu}+\Theta+\nu^*f^*H$ is semi-ample and 
that 
$(X, \Delta+f^*H)$ is slc. 
By Proposition \ref{prop3}, 
$K_X+\Delta+f^*H$ is semi-ample. 
In particular, $K_X+\Delta+f^*H$ is $f$-semi-ample. 
Then $K_X+\Delta$ is $f$-semi-ample. 
\end{proof}

By the same arguments as in the proof of Theorem \ref{thm10} 
(resp.~Corollary \ref{ko-slc}), 
we obtain the following theorem (resp.~corollary), which is a 
relative version of 
Theorem \ref{thm1} (resp.~Corollary \ref{cor44}). 

\begin{thm}\label{thm11}
Let $f:X\to Y$ be a projective 
morphism between projective varieties. 
Let $(X, \Delta)$ be an lc pair such that 
$K_X+\Delta$ is $f$-nef and $f$-log abundant. 
Then $K_X+\Delta$ is $f$-semi-ample.  
\end{thm}

\begin{cor}\label{cor412} 
Let $f:X\to Y$ be 
a projective 
morphism from a projective slc pair $(X, \Delta)$ to a {\em{(}}not necessarily 
irreducible{\em{)}} projective variety $Y$. 
Let $\nu:X^\nu\to X$ be the normalization such that 
$K_{X^\nu}+\Theta=\nu^*(K_X+\Delta)$. 
Assume that $K_{X^\nu}+\Theta$ is nef and log abundant over $Y$. Then 
$K_X+\Delta$ is $f$-semi-ample. 
\end{cor}

\subsection{Relative abundance conjecture}\label{subsec41} 
In this subsection, we make some remarks on the relative abundance conjecture. 

After we circulated this paper, Hacon and Xu proved the relative version of Theorem 
\ref{main-thm2} (cf.~Proposition \ref{prop3}) in \cite{hacon-xu-slc}.

\begin{thm}[{cf.~\cite[Theorem 1.4]{hacon-xu-slc}}]\label{thm-hx} 
Let $(X, \Delta)$ be an slc pair, let $f:X\to Y$ be a projective 
morphism onto an algebraic variety $Y$, and let $\nu:X^\nu\to X$ be the normalization with 
$K_{X^\nu}+\Theta=\nu^*(K_X+\Delta)$. 
If $K_{X^\nu}+\Theta$ is semi-ample over $Y$, then 
$K_X+\Delta$ is semi-ample over $Y$. 
\end{thm}

The proof of \cite[Theorem 1.4]{hacon-xu-slc} in \cite[Section 4]{hacon-xu-slc} depends 
on Koll\'ar's gluing theory (see, for example, \cite{kollar-source}, 
\cite{kollar-book}, and \cite{hacon-xu-lc-closure}) and the finiteness of 
the log pluricanonical representation:~Theorem \ref{main-thm}. 
Note that Hacon and Xu prove a slightly weaker version of Theorem \ref{main-thm} 
in \cite[Section 3]{hacon-xu-slc}, which is 
sufficient for the proof of \cite[Theorem 1.4]{hacon-xu-slc}. 
Their arguments in \cite[Section 3]{hacon-xu-slc} 
are more Hodge theoretic than ours and 
use the finiteness result in the case when the Kodaira 
dimension is zero established in \cite{gongyo-aban}. 
We note that Theorem \ref{thm-hx} implies the relative versions (or generalizations) of 
Theorem \ref{thm1}, Corollary \ref{cor44}, Corollary \ref{45cor}, Theorem \ref{thm2}, Theorem \ref{thm4}, 
Theorem \ref{thm10}, Corollary \ref{ko-slc}, 
Theorem \ref{thm11}, and Corollary \ref{cor412} 
without assuming the projectivity of varieties. 
We leave the details as exercises for the reader. 
 
\subsection{Miscellaneous applications}\label{mis} 
In this subsection, we collect some miscellaneous applications 
related to the base point free theorem and the abundance conjecture. 

The following theorem is the log canonical version of Fukuda's 
result. 

\begin{thm}[{cf.~\cite[Theorem 0.1]{fukuda2}}]\label{thm414}  
Let $(X, \Delta)$ be a projective 
lc pair. 
Assume that $K_X+\Delta$ is numerically equivalent 
to some semi-ample 
$\mathbb Q$-Cartier $\mathbb Q$-divisor 
$D$. 
Then $K_X+\Delta$ is 
semi-ample. 
\end{thm}

\begin{proof}
By taking a dlt blow-up (cf.~Theorem \ref{dltblowup}), we can assume that 
$(X, \Delta)$ is dlt. 
By the induction on dimension and 
Proposition \ref{prop3}, we have that 
$(K_X+\Delta)|_{\llcorner \Delta\lrcorner}$ is 
semi-ample. 
By \cite[Theorem 1.1]{fujino-bpf}, we can prove 
the semi-ampleness of $K_X+\Delta$. 
For the details, see the proof of 
\cite[Theorem 6.3]{gongyo-aban}. 
\end{proof}

By using the deep result in \cite{ckp}, we have a 
slight generalization of Theorem \ref{thm414} and \cite[Corollary 3]{ckp}. 
It is also a generalization of Theorem \ref{thm1}. 

\begin{thm}[{cf.~\cite[Corollary 3]{ckp}}]\label{thm415}
Let $(X, \Delta)$ be a projective lc pair 
and let $D$ be a $\mathbb Q$-Cartier $\mathbb Q$-divisor 
on $X$ such that 
$D$ is nef and log abundant with respect to 
$(X, \Delta)$. 
Assume that $K_X+\Delta\equiv D$. 
Then $K_X+\Delta$ is semi-ample. 
\end{thm}

\begin{proof}
By replacing $(X, \Delta)$ with its dlt blow-up 
(cf.~Theorem \ref{dltblowup}), 
we can assume that $(X, \Delta)$ 
is dlt. 
Let $f:Y\to X$ be a log resolution. 
We put $K_Y+\Delta_Y=f^*(K_X+\Delta)+F$ with 
$\Delta_Y=f_*^{-1}\Delta+\sum E$ where 
$E$ runs through all the $f$-exceptional 
prime divisors on $Y$. 
We note that $F$ is effective and $f$-exceptional. 
By \cite[Corollary 1]{ckp}, 
\begin{align*}
\kappa (X, K_X+\Delta)=\kappa (Y, K_Y+\Delta_Y)\geq 
\kappa (Y, f^*D+F)=\kappa (X, D). 
\end{align*} 
By the assumption, $\kappa (X, D)=\nu (X, D)=\nu (X, K_X+\Delta)$. 
On the other hand, $\nu(X, K_X+\Delta)\geq \kappa (X, K_X+\Delta)$ 
always holds. Therefore, 
$\kappa (X, K_X+\Delta)=\nu (X, K_X+\Delta)$, that is, 
$K_X+\Delta$ is nef and abundant. 
By applying the above argument to every 
lc center of $(X, \Delta)$, we 
obtain that $K_X+\Delta$ is nef and log abundant. 
Thus, by Theorem \ref{thm1}, 
we obtain that $K_X+\Delta$ is semi-ample.  
\end{proof}

\begin{rem}\label{rem416} 
By the proof of Theorem \ref{thm415}, 
we see that we can weaken the assumption 
as follows. 
Let $(X, \Delta)$ be a projective lc pair. 
Assume that $K_X+\Delta$ is numerically equivalent to 
a nef and abundant $\mathbb Q$-Cartier $\mathbb Q$-divisor and 
that $\nu_W^*((K_X+\Delta)|_W)$ is numerically equivalent to 
a nef and abundant $\mathbb Q$-Cartier $\mathbb Q$-divisor 
for every lc center $W$ of $(X, \Delta)$, 
where $\nu_W:W^\nu\to W$ is the normalization of 
$W$. 
Then $K_X+\Delta$ is semi-ample. 
\end{rem}

Theorem \ref{thm-m3} is a 
generalization of \cite[Theorem 1.7]{gongyo-weak}. 
The proof is the same as \cite[Theorem 1.7]{gongyo-weak} 
once we adopt \cite[Theorem 1.1]{fujino-bpf}. 

\begin{thm}[{cf.~\cite[Theorems 6.4, 6.5]{gongyo-aban}}]\label{thm-m3} 
Let $(X, \Delta)$ be a projective lc pair such that 
$-(K_X+\Delta)$ {\em{(}}resp.~$K_X+\Delta${\em{)}} 
is nef and abundant. 
Assume that 
$\dim \Nklt(X, \Delta)\leq 1$ where 
$\Nklt (X, \Delta)$ is the non-klt locus of the pair $(X, \Delta)$. 
Then $-(K_X+\Delta)$ {\em{(}}resp.~$K_X+\Delta${\em{)}} is 
semi-ample. 
\end{thm}

\begin{proof}
Let $T$ be the non-klt locus of $(X, \Delta)$. 
By the same argument as in the proof of \cite[Theorem 3.1]{gongyo-weak}, 
we can check that $-(K_X+\Delta)|_T$ (resp.~$(K_X+\Delta)|_T$) 
is semi-ample. 
Therefore, $-(K_X+\Delta)$ (resp.~$K_X+\Delta$) is 
semi-ample by \cite[Theorem 1.1]{fujino-bpf}. 
\end{proof}

Similarly, we can prove Theorem \ref{thm-m2}. 

\begin{thm}\label{thm-m2}
Let $(X, \Delta)$ be a projective lc pair. 
Assume that $-(K_X+\Delta)$ is nef and abundant 
and that $(K_X+\Delta)|_W\equiv 0$ for every 
lc center $W$ of $(X, \Delta)$. 
Then $-(K_X+\Delta)$ is semi-ample. 
\end{thm}

\begin{proof}
By taking a dlt blow-up (cf.~Theorem \ref{dltblowup}), 
we can assume that $(X, \Delta)$ is dlt. 
By \cite[Theorem 1.2]{gongyo-aban} 
(cf.~Corollary \ref{ko-slc}), $(K_X+\Delta)|_{\llcorner \Delta\lrcorner}$ is 
semi-ample. Therefore, $K_X+\Delta$ is semi-ample by 
\cite[Theorem 1.1]{fujino-bpf}. 
\end{proof}

\section{Non-vanishing, abundance, and extension conjectures}\label{sec5}

In this final section, we discuss the relationship 
among various conjectures in the minimal model program. 
Roughly speaking, we prove that the abundance conjecture 
for projective log canonical pairs (cf.~Conjecture \ref{abun}) 
is equivalent to 
the non-vanishing conjecture 
(cf.~Conjecture \ref{nonvan4}) and 
the extension conjecture for projective 
dlt pairs (cf.~Conjecture \ref{c-dlt}). 

First, let us recall the weak non-vanishing conjecture for 
projective lc pairs (cf.~\cite[Conjecture 1.3]{birkar}). 

\begin{conj}[Weak non-vanishing conjecture]\label{nonvan3}
Let $(X, \Delta)$ be a projective 
lc pair such that $\Delta$ is an $\mathbb R$-divisor. 
Assume that $K_X+\Delta$ is pseudo-effective. 
Then there exists an effective $\mathbb R$-divisor $D$ on 
$X$ such that $K_X+\Delta\equiv D$. 
\end{conj}

Conjecture \ref{nonvan3} is known to be one of the most important 
problems in the minimal model theory (cf.~\cite{birkar}). 

\begin{rem}\label{rem10}
By \cite[Theorem 1]{ckp}, $K_X+\Delta\equiv D\geq 0$ in 
Conjecture \ref{nonvan3} means that 
there is an effective $\mathbb R$-divisor $D'$ such 
that $K_X+\Delta\sim_{\mathbb R}D'$. 
\end{rem}

By Remark \ref{rem10} and Lemma \ref{lem11} below, 
Conjecture \ref{nonvan3} in dimension $\leq n$ is 
equivalent to Conjecture 1.3 of \cite{birkar} 
in dimension $\leq n$ 
with the aid of dlt blow-ups (cf.~Theorem \ref{dltblowup}). 

\begin{lem}\label{lem11} 
Assume that {\em{Conjecture \ref{nonvan3}}} holds in dimension $\leq n$. 
Let $f:X\to Z$ be a projective morphism between quasi-projective varieties 
with $\dim X=n$. 
Let $(X, \Delta)$ be an lc pair such that 
$K_X+\Delta$ is pseudo-effective over $Z$. 
Then there exists an effective $\mathbb R$-Cartier $\mathbb R$-divisor $M$ on $X$ 
such that $K_X+\Delta\sim _{\mathbb R, Z}M$. 
\end{lem}
\begin{proof}
Apply Conjecture \ref{nonvan3} and Remark \ref{rem10} to 
the generic fiber of $f$. 
Then, by \cite[Lemma 3.2.1]{bchm}, we obtain $M$ with the required 
properties. 
\end{proof}

We give a small remark on Birkar's paper \cite{birkar}. 

\begin{rem}[Absolute versus relative]Let $f:X\to Z$ be 
a projective morphism between {\em{projective}} varieties. 
Let $(X, B)$ be a $\mathbb Q$-factorial 
dlt pair and let $(X, B+C)$ be an lc pair such that 
$C\geq 0$ and that $K_X+B+C$ is nef over $Z$. 
Let $H$ be a very ample Cartier divisor on $Z$. 
Let $D$ be a general member of $|2(2\dim X+1)H|$. 
In this situation, $(X, B+\frac{1}{2}f^*D)$ is dlt, 
$(X, B+\frac{1}{2}f^*D+C)$ is lc, and 
$K_X+B+\frac{1}{2}f^*D+C$ is nef by Kawamata's bound on the 
length of extremal rays. 
The minimal model program on $K_X+B+\frac{1}{2}f^*D$ with 
scaling of $C$ is the minimal model program on $K_X+B$ {\em{over 
$Z$}} with scaling of $C$. By this observation, the arguments in 
\cite{birkar} work without appealing relative settings 
if the considered varieties are {\em{projective}}. 
\end{rem}

\begin{thm}\label{lcabun}
The abundance theorem for projective 
klt pairs in dimension $\leq n$ and 
{\em{Conjecture \ref{nonvan3}}} for projective 
$\mathbb Q$-factorial dlt pairs in dimension $\leq n$  imply 
the abundance theorem for projective lc pairs in dimension  $\leq n$.
\end{thm}

\begin{proof} 
Let $(X, \Delta)$ be an $n$-dimensional projective lc pair such that 
$K_X+\Delta$ is nef. 
As we explained in \ref{27}, by \cite[Theorems 1.4, 1.5]{birkar}, 
the minimal model 
program with ample scaling terminates for 
projective $\mathbb Q$-factorial klt pairs in dimension $\leq n$. 
Moreover, we can assume that 
$(X, \Delta)$ is a projective $\mathbb Q$-factorial 
dlt pair by taking 
a dlt blow-up 
(cf.~Theorem \ref{dltblowup}). 
Thus, by Theorem \ref{thm4}, we obtain the desired 
result. 
\end{proof}

The following corollary is a result on a generalized abundance conjecture 
formulated by Nakayama's numerical Kodaira dimension $\kappa _{\sigma}$. 

\begin{cor}[Generalized abundance conjecture]\label{cor14} 
Assume that the abundance conjecture 
for projective klt pairs in dimension $\leq n$ and {\em{Conjecture 
\ref{nonvan3}}} for $\mathbb Q$-factorial dlt pairs in dimension $\leq n$. 
Let $(X, \Delta)$ be an $n$-dimensional 
projective lc pair.  
Then $\kappa (X, K_X+\Delta)=\kappa _{\sigma}(X, K_X+\Delta)$. 
\end{cor}

\begin{proof}
We can assume that $(X, \Delta)$ is a $\mathbb Q$-factorial projective dlt 
pair by replacing it with its dlt blow-up (cf.~Theorem \ref{dltblowup}). 
Let $H$ be a general effective sufficiently ample Cartier 
divisor on $X$. 
We can run the minimal model program 
with scaling of $H$ by \ref{27}.  
Then we obtain a good minimal model by Theorem \ref{lcabun} 
if $K_X+\Delta$ is pseudo-effective. 
When $K_X+\Delta$ is not pseudo-effective, 
we have a Mori fiber space structure. 
In each step of the minimal model program, 
$\kappa$ and $\kappa _{\sigma}$ are preserved. 
So, we obtain $\kappa (X, K_X+\Delta)=\kappa _{\sigma}(X, K_X+\Delta)$. 
\end{proof}

Finally, we explain the importance of 
Theorem \ref{main-thm2} toward the abundance conjecture. 
Let us consider the following two conjectures. 

\begin{conj}[Non-vanishing conjecture]\label{nonvan4}
Let $(X, \Delta)$ be a projective 
lc pair such that $\Delta$ is an $\mathbb R$-divisor. 
Assume that $K_X+\Delta$ is pseudo-effective. 
Then there exists an effective $\mathbb R$-divisor $D$ on 
$X$ such that $K_X+\Delta\sim_{\mathbb{R}} D$. 
\end{conj}

As we pointed it out in Remark \ref{rem10}, 
Conjecture \ref{nonvan4} in dimension $n$ follows from Conjecture \ref{nonvan3} in dimension $n$. 
For related topics on the non-vanishing conjecture, see \cite[Section 8]{dhp} and \cite{g4}, 
where Conjecture \ref{nonvan4} is reduced to the case when 
$X$ is a smooth projective variety and 
$\Delta=0$ by assuming the global ACC conjecture and 
the ACC for log canonical thresholds 
(see \cite[Conjecture 8.2 and Conjecture 8.4]{dhp}). 

\begin{conj}[{Extension conjecture for dlt pairs (cf.~\cite[Conjecture 1.3]{dhp}}]
\label{c-dlt} 
Let $(X, S+B)$ be an $n$-dimensional projective 
dlt pair such that $B$ is an effective $\mathbb Q$-divisor, 
$\llcorner S+B \lrcorner =S$, $K_X+S+B$ is nef, 
and $K_X+S+B\sim _{\mathbb Q}D\geq 0$ where $S\subset\Supp D$.
Then $$H^0(X,\mathcal O _X(m(K_X+S+B)))\to H^0(S,  \mathcal O _S(m(K_X+S+B)))  $$
is surjective for all sufficiently divisible integers $m\geq 2$. 
\end{conj}

In Conjecture \ref{c-dlt}, if $(X, S+B)$ is a plt pair, 
equivalently, $S$ is normal,  
then it is true by \cite[Corollary 1.8]{dhp}.

The following theorem is essentially 
contained in the proof of \cite[Theorem 1.4]{dhp}. 
However, our proof of Theorem \ref{coj implication} 
is slightly different from the arguments in \cite{dhp} 
because we directly use Birkar's result \cite[Theorems 1.4 and 1.5]{birkar} and Theorem \ref{main-thm2}. 

\begin{thm}\label{coj implication} 
Assume that {\em{Conjecture \ref{nonvan4}}} and 
{\em{Conjecture \ref{c-dlt}}} hold true in dimension $\leq n$. 
Then {\em{Conjecture \ref{abun}}} is true in dimension $n$. 
\end{thm}

\begin{proof} 
By Theorem \ref{main-thm2}, it is sufficient to treat log canonical pairs. 
This reduction is crucial for our inductive proof. 
We show the statement by induction on dimension. 
Note that we can freely use the minimal model program with ample scaling 
for projective dlt pairs  by Conjecture \ref{nonvan3} and 
Birkar's results (cf.~\cite[Theorems 1.4 and 1.5]{birkar} and \ref{27}). 
By Theorem \ref{thm1} and Corollary \ref{cor14}, it is sufficient to show the abundance conjecture 
for an $n$-dimensional projective kawamata log terminal pair $(X,\Delta)$ 
such that $K_X+\Delta$ is pseudo-effective. 
By Conjecture \ref{nonvan4}, we see $\kappa(X, K_X+\Delta) \geq 0$ 
(cf.~\cite[Corollary 2.1.4]{choi}). 
By Kawamata's well-known inductive argument (cf.~\cite[Theorem 7.3]{kawamata_pluri}), 
we may assume that $\kappa(X, K_X+\Delta)=0$. 
We take an effective divisor $D$ such that $D \sim_{\mathbb{Q}}K_X+\Delta$. 
We take a resolution $\varphi:Y \to X$ such that 
$\Exc (\varphi)\cup \Supp f_*^{-1}(\Delta+D)$ is a simple normal crossing divisor on $Y$. 
Let $B$ and $E$ be effective $\mathbb{Q}$-divisors satisfying:
$$ K_Y+B=\varphi^*(K_X+\Delta)+E,
$$
and $E$ and $B$ have no common irreducible components. 
Now we know that $\kappa(X, K_X+\Delta)=\kappa(Y, K_Y+B)=0$ and 
$\kappa_{\sigma}(X, K_X+\Delta)=\kappa_{\sigma}(Y, K_Y+B)$. 
Thus, by replacing $X$ with $Y$, we may further assume that $X$ is smooth and 
$\Delta+D$ has a simple normal crossing support. Let 
$$\Delta=\sum \delta_i D_i\ \text{and}\ D=\sum d_iD_i$$
be the irreducible decompositions. 
We put 
$$\Delta'=\Delta-\sum_{d_i \not =0}\delta_i D_i + D_{\mathrm{red}}, $$
where $D_{\mathrm{red}}=\sum _{d_i\ne 0}D_i$.   
Then the effective divisor $\Delta'$ satisfies 
$$\Supp \Delta' =\Supp (\Delta+D),\ \Supp\llcorner \Delta' \lrcorner= \Supp D,$$ 
and $\Supp D=\Supp (\Delta'-\Delta)$ since $(X,\Delta)$ is klt. 
Note that 
$$\kappa(X, K_X+\Delta)= \kappa(X, K_X+\Delta')
$$ and 
$$
\kappa_{\sigma}(X, K_X+\Delta)=\kappa_{\sigma}
(X, K_X+\Delta').$$ 
We take a minimal model 
$$f:(X,\Delta') \dashrightarrow (Y,\Gamma')$$ 
of $(X,\Delta')$. If $(Y,\Gamma')$ is klt, then $\llcorner \Delta' \lrcorner$ is 
$f$-exceptional. Thus we have $K_Y+\Gamma' \sim_{\mathbb{Q}}0$ since 
$\Supp\llcorner \Delta' \lrcorner= \Supp D.$
Therefore,  
$$\kappa_{\sigma}(X, K_X+\Delta)=\kappa_{\sigma}(X, K_X+\Delta')=0.$$ 
It is a desired result. 
So, from now on, we assume that $S:=\llcorner \Gamma' \lrcorner \not =0$. 
Then, by Conjecture \ref{c-dlt}, it holds that 
$$H^0(Y,\mathcal O _Y(m(K_Y+\Gamma')))\to H^0(S,  \mathcal O _S(m(K_Y+\Gamma')))  $$
is surjective for all sufficiently divisible integers $m\geq 2$. 
By the hypothesis of the induction, 
it holds that 
$K_S+\Gamma_S=(K_Y+\Gamma')|_{S}$ is semi-ample. 
Note that the pair $(S, \Gamma _S)$ is an sdlt pair.  
In particular, 
$H^0(S,  \mathcal O _S(m(K_Y+\Gamma'))) \not=0$. 
However, since $\Supp \llcorner \Delta' \lrcorner= \Supp D$ and $\kappa(Y, K_Y+\Gamma')=0$, 
$$H^0(Y,\mathcal O _Y(m(K_Y+\Gamma')))\to H^0(S,  \mathcal O _S(m(K_Y+\Gamma')))$$ 
is a zero map.  It is a contradiction. 
Thus we see that $S=0$. 
Therefore, we obtain $\kappa(X, K_X+\Delta)=\kappa_{\sigma}(X, K_X+\Delta)$. 
\end{proof}

We have a generalization 
of \cite[Theorem 1.4]{dhp} as a corollary of Theorem \ref{coj implication}. 
For a different approach to the existence of good minimal models, see 
\cite{birkar2}. 

\begin{cor}[{cf.~\cite[Theorem 1.4]{dhp}}]\label{general14}
Assume that {\em{Conjecture \ref{nonvan4}}} and {\em{Conjecture \ref{c-dlt}}} 
hold true in dimension $\leq n$. 
Let $f:X\to Y$ be a projective 
morphism between quasi-projective varieties. 
Assume that $(X, \Delta)$ is an $n$-dimensional 
dlt pair such that 
$K_X+\Delta$ is pseudo-effective over $Y$. 
Then 
$(X, \Delta)$ has a good minimal model 
$(X', \Delta')$ over 
$Y$. 
\end{cor}
\begin{proof}
By Conjecture \ref{nonvan4} with Lemma \ref{lem11}, 
we can run the minimal model program with ample scaling 
(cf.~\ref{27}). 
Therefore, we can construct a minimal model 
$(X', \Delta')$ over $Y$. 
By Theorem \ref{coj implication}, 
$K_{X'}+\Delta'$ is semi-ample when $Y$ is a point and 
$\Delta'$ is a $\mathbb Q$-divisor. 
By the relative version of Theorem \ref{thm1} and the induction on dimension, 
we can check that $K_{X'}+\Delta'$ is semi-ample over $Y$ when 
$\Delta'$ is a $\mathbb Q$-divisor.  
If $\Delta'$ is an $\mathbb R$-divisor, 
then we can reduce it to the case when $\Delta'$ is a $\mathbb Q$-divisor 
by using Shokurov's polytope (see, for example, the proof of Theorem 3.1 in \cite{fujino-gongyo}) 
and obtain that $K_{X'}+\Delta'$ is semi-ample over $Y$. 
It is a standard argument. 
\end{proof}

We close this paper with an easy observation. 
Conjecture \ref{c-dlt} follows from Conjecture \ref{abun} 
by a cohomology injectivity theorem. 

\begin{prop}\label{inj imply ext} 
Assume that {\em{Conjecture \ref{abun}}} is 
true in dimension $n$. Then {\em{Conjecture \ref{c-dlt}}} holds true in dimension $n$.
More precisely, in {\em{Conjecture \ref{c-dlt}}}, if $K_X+S+B$ is semi-ample, 
then the restriction map is surjective for every $m\geq 2$ such that 
$m(K_X+S+B)$ is Cartier. 
\end{prop}

\begin{proof}
Let $f:Y\to X$ be a projective birational morphism from a smooth variety $Y$ such that 
$\Exc (f)\cup \Supp f_*^{-1}(S+B)$ is a simple normal crossing divisor on $Y$ and 
that $f$ is an isomorphism over 
the generic point of every lc center of the pair $(X, S+B)$. 
Then we can write 
$$
K_Y+S'+F=f^*(K_X+S+B)+E
$$ 
where 
$S'$ is the strict transform of $S$, $E$ is an effective Cartier divisor, 
$F$ is an effective $\mathbb Q$-divisor 
with $\llcorner F\lrcorner=0$. 
Note that $E$ is $f$-exceptional. 
We consider the short exact sequence 
$$
0\to \mathcal O_Y(E-S')\to \mathcal O_Y(E)\to \mathcal O_{S'}(E)\to 0. 
$$ 
By the relative Kawamata--Viehweg vanishing theorem, 
$R^if_*\mathcal O_Y(E-S')=0$ for every $i>0$. 
Therefore, 
$$
\mathcal O_X\simeq f_*\mathcal O_Y(E)\to f_*\mathcal O_{S'}(E)
$$ 
is surjective. 
Thus, we obtain $f_*\mathcal O_{S'}(E)\simeq \mathcal O_S$. 
Let $m$ be a positive integer such that 
$m(K_X+S+B)$ is Cartier with $m\geq 2$. 
We put $L=m(K_X+S+B)$. 
It is sufficient to prove that 
$$
H^0(Y, \mathcal O_Y(f^*L+E))\to H^0(S', \mathcal O_{S'}(f^*L+E))
$$ 
is surjective. 
By Conjecture \ref{abun}, $K_X+S+B$ is semi-ample. 
Let $g:X\to Z$ be the Iitaka fibration associated to 
$K_X+S+B$. 
Then there is an ample $\mathbb Q$-Cartier 
$\mathbb Q$-divisor $A$ on $Z$ such that 
$K_X+S+B\sim _{\mathbb Q}g^*A$. We note that 
\begin{align*}
(f^*L+E-S')-(K_Y+F)&=(m-1)f^*(K_X+S+B)\\
&\sim _{\mathbb Q}(m-1)f^*g^*A. 
\end{align*} 
Since $S\subset \Supp D$ and 
$K_X+S+B\sim _{\mathbb Q}D\geq 0$, we have 
$g\circ f(S')\subsetneq Z$. 
Note that $g:X\to Z$ is the Iitaka fibration associated to $K_X+S+B$. 
Therefore, it is easy to see 
that 
$$
H^i(Y, \mathcal O_Y(f^*L+E-S'))\to H^i(Y, \mathcal O_Y(f^*L+E))
$$ 
is injective for every $i$ because 
$|kf^*g^*A-S'|\ne \emptyset$ for $k\gg 0$ 
(see, for example, \cite[Theorem 6.1]{F-fund}). 
In particular, 
$$
H^1(Y, \mathcal O_Y(f^*L+E-S'))\to H^1(Y, \mathcal O_Y(f^*L+E))
$$ 
is injective. Thus we obtain that 
$$
H^0(Y, \mathcal O_Y(f^*L+E))\to H^0(S', \mathcal O_{S'}(f^*L+E))
$$ 
is surjective. 
It implies the desired surjection 
$$
H^0(X, \mathcal O_X(m(K_X+S+B)))\to H^0(S, \mathcal O_S(m(K_X+S+B))). 
$$
We finish the proof. 
\end{proof}

Proposition \ref{inj imply ext} shows that Conjecture \ref{c-dlt} is a reasonable conjecture in the 
minimal model theory. 


\end{document}